\documentclass[letterpaper, 11pt]{amsart}
\usepackage[english]{babel}
\usepackage{graphicx,xcolor}
\usepackage[all]{xy}
\usepackage{amsfonts,dsfont,mathrsfs}
\usepackage{amsmath, amsthm, amssymb}
\usepackage{mathrsfs}
\usepackage{enumerate, url}
\usepackage[foot]{amsaddr}
\usepackage{fancyhdr}
\usepackage{hyperref}
\usepackage{ifthen,srcltx}
\usepackage{comment}

\newcommand{\G}{\ensuremath{\mathbb{G}}}

\def\bs{\backslash}

\newcommand{\bbA}{\ensuremath{\mathbb{A}}}
\newcommand{\bbC}{\ensuremath{\mathbb{C}}}

\newcommand{\bbG}{\ensuremath{\mathbb{G}}}
\newcommand{\bbH}{\ensuremath{\mathbb{H}}}

\newcommand{\bbQ}{\ensuremath{\mathbb{Q}}}
\newcommand{\bbR}{\ensuremath{\mathbb{R}}}

\newcommand{\bbZ}{\ensuremath{\mathbb{Z}}}

\newcommand{\frg}{\ensuremath{\mathfrak{g}}}
\newcommand{\frh}{\ensuremath{\mathfrak{h}}}

\newcommand{\fra}{\ensuremath{\mathfrak{a}}}

\newcommand{\calO}{\ensuremath{\mathcal{O}}}

\DeclareMathOperator{\Id}{Id}

\DeclareMathOperator{\rk}{rk}
\DeclareMathOperator{\rg}{rg}

\DeclareMathOperator{\Isom}{Isom}

\DeclareMathOperator{\Aut}{Aut}

\DeclareMathOperator{\GL}{GL}
\DeclareMathOperator{\PGL}{PGL}
\DeclareMathOperator{\SL}{SL}
\DeclareMathOperator{\PSL}{PSL}
\DeclareMathOperator{\Sp}{Sp}

\DeclareMathOperator{\SO}{SO}
\DeclareMathOperator{\PO}{PO}

\DeclareMathOperator{\SU}{SU}

\DeclareMathOperator{\Gr}{Gr}

\DeclareMathOperator{\Comm}{Comm}
\DeclareMathOperator{\Gal}{Gal}

\DeclareMathOperator{\pr}{pr}

\DeclareMathOperator{\Ad}{Ad}
\DeclareMathOperator{\ad}{ad}
\def\bs{\backslash}

\newcommand{\natls}{{\mathbb N}}

\newcommand\FF{{\mathcal F}}

\newcommand\LL{{\mathcal L}}
\newcommand\MM{{\mathcal M}}

\newcommand\PP{{\mathcal P}}

\newcommand\PMF{{\PP\kern-2pt\MM\FF}}

\newcommand\PML{{\PP\kern-2pt\MM\LL}}

\newcommand\Z{{\mathbb Z}}
\newcommand\R{{\mathbb R}}

\newcommand\Q{{\mathbb Q}}

\newcommand{\fsubd}{\mathrel{{\scriptstyle\searrow}\kern-1ex^d\kern0.5ex}}
\newcommand{\bsubd}{\mathrel{{\scriptstyle\swarrow}\kern-1.6ex^d\kern0.8ex}}
\newcommand{\fsubeq}{\mathrel{\raise-.7ex\hbox{$\overset{\searrow}{=}$}}}
\newcommand{\bsubeq}{\mathrel{\raise-.7ex\hbox{$\overset{\swarrow}{=}$}}}

\newcommand{\tsh}[1]{\left\{\kern-.9ex\left\{#1\right\}\kern-.9ex\right\}}

\newtheorem{thm}{Theorem}[section]
\newtheorem{prop}[thm]{Proposition}
\newtheorem{lemma}[thm]{Lemma}
\newtheorem{cor}[thm]{Corollary}
\theoremstyle{definition}
\newtheorem{dfn}[thm]{Definition}

\newtheorem{rem}[thm]{Remark}

\newtheorem{problem}[thm]{Problem}

\newtheorem{conj}[thm]{Conjecture}

\newtheorem{question}[thm]{Question}
\newtheorem{hypothesis}[thm]{Hypothesis}
\newtheorem{sassume}[thm]{Standing assumptions}
\newtheorem{assume}[thm]{Assumptions}
\newtheorem{notation}[thm]{Notation}

\newsavebox{\commentbox}
\newenvironment{mycomment}%
{\ifthenelse{\equal{\showcomments}{yes}}%
{\footnotemark
        \begin{lrbox}{\commentbox}
        \begin{minipage}[t]{1.25in}\raggedright\sffamily\tiny
        \footnotemark[\arabic{footnote}]}
{\begin{lrbox}{\commentbox}}}
{\ifthenelse{\equal{\showcomments}{yes}}
{\end{minipage}\end{lrbox}\marginpar{\usebox{\commentbox}}}
{\end{lrbox}}}

\newcommand{\showcomments}{yes}

\begin{document}

\title[The Greenberg-Shalom Hypothesis]{Greenberg-Shalom's Commensurator Hypothesis and Applications}

\author[N. Brody]{Nic Brody}
\address{Department of Mathematics \\ University of California at Santa Cruz \\ Santa Cruz, CA, USA}

\author[D. Fisher]{David Fisher}
\address{Department of Mathematics \\ Rice University \\ Houston, TX, USA }

\author[M. Mj]{Mahan Mj}
\address{School of Mathematics \\ Tata Institute of Fundamental Research \\ Mumbai, India}

\author[W. van Limbeek]{Wouter van Limbeek}
\address{Department of Mathematics, Statistics, and Computer Science \\ University of Illinois at Chicago \\ Chicago, IL, USA}

\date{\today}

\begin{abstract}
We discuss many surprising implications of a positive answer to a question raised in some cases by Greenberg
in the {$1970$s} and  more generally by Shalom in the early $2000$s.  We refer to this positive answer as the Greenberg-Shalom hypothesis.  This hypothesis then says that any infinite discrete subgroup of a semisimple Lie group with dense commensurator is a lattice in a product of some factors. For some applications it is natural to extend the hypothesis to cover semisimple algebraic groups over other fields as well.
\end{abstract}

\maketitle

\setcounter{tocdepth}{1}
\numberwithin{equation}{section}
\tableofcontents

\section{Introduction}
\label{sec:intro}


Let $G$ be a real or $p$-adic semisimple Lie group with finite center and without compact factors, or a finite product of such groups.  More precisely, we consider $G=\G(k)$ where $k$ is a local field of characteristic zero and $\G$ is either a semisimple algebraic group defined over $k$ or { a product of semisimple algebraic groups over possibly different fields of this type.} Let $\Gamma\subseteq G$ be a discrete subgroup with commensurator $\Delta$. Borel proved that if $\Gamma$ is an arithmetic lattice, then the commensurator contains the rational points of $G$ and in particular is almost dense in $G$ \cite{borel-comm}. Here, `almost dense' means its closure has finite index in $G$. Greenberg (for $G=\SO(n,1)$) and Shalom (in general) asked whether almost dense commensurator of a discrete, Zariski dense subgroup $\Gamma$ detects its arithmeticity.  We formulate the positive answer to this question as a hypothesis, since here we are primarily interested in the implications if it is true.

\begin{hypothesis}[{Greenberg \cite{greenberg-comm}, Shalom, see \cite{llr-comm}}] Let $G$ be a semisimple Lie group with finite center and without compact factors. Suppose $\Gamma\subseteq G$ is a discrete, Zariski dense subgroup of $G$ whose commensurator $\Delta\subseteq G$ is almost dense. Then $\Gamma$ is an arithmetic lattice in $G$.
\label{q:shalom}\end{hypothesis}

We will refer to Hypothesis \ref{q:shalom} as the \emph{Greenberg-Shalom hypothesis}.  We remark here that the hypothesis seems quite plausible on first encounter, and Greenberg and Shalom both seem inclined to believe it. On the other hand, the many implications of the hypothesis discussed here may cast some doubt on the likelihood that it is correct. The purpose of this article is to exhibit these connections between the hypothesis and other problems, many of which do not a priori involve commensurators. Taken together they indicate a point of view from which one may at least view the problems as connected, whether one views it as evidence for or against the hypothesis or for or against any of the results the hypothesis implies.  A short overview of these applications is given below.

	\begin{enumerate}[(1)]
		\item[\eqref{sec:irred2latt}] A question of Benoist on existence of discrete, irreducible free or surface subgroups in products of simple real and $p$-adic Lie groups {(see Corollaries~\ref{noirreduciblesurfacesinpadics} and \ref{cor:algebraicsurfacegroups})}.
		\item[\eqref{sec:lukk}] A conjecture of Lyndon-Ullman and Kim-Koberda on groups generated by parabolics {(see Theorem~\ref{thm:lukk})}.
		\item[\eqref{sec:arith-3dim}] Existence of elements with integral traces in hyperbolic 3-manifold groups {(see Theorem~\ref{thm:3manifold})}.
		
		\item[\eqref{sec:cohere}] A question of Serre on coherence of $\SL(2,\bbZ[1/p])$ and related groups, and a problem of Wise whether coherence is geometric {(see Theorem~\ref{thm-cohere})}.
		\item[\eqref{sec:mz}] The Margulis-Zimmer conjecture on arithmeticity of commensurated subgroups of $S$-arithmetic lattices {(see Theorem~\ref{thm-mz})}.

	\end{enumerate}
	
	{We believe that one of the main sources of
motivation for  Shalom's question \cite{llr-comm} that led to Hypothesis~\ref{q:shalom}  is the Margulis-Zimmer Conjecture~\ref{conj:mz} mentioned in Item (5) above. The Margulis-Zimmer Conjecture \cite{shalom-willis} seeks to classify commensurated subgroups
of higher rank lattices.
In Remark~\ref{rmk:mz}(3), we shall spell out this connection.}

{
A key impetus for this paper  is a close connection between commensurators and irreducibility. While this idea has been commonplace
for lattices, it does not seem to have been noticed in this context before.   In section \ref{sec:assume} we lay out some key definitions for irreducibility and then
make the key connection in Lemma \ref{lem:intersectionprojection}.  In Proposition \ref{prop-irred2latt} this
is used to show that the Greenberg-Shalom hypothesis implies that many irreducible groups are automatically lattices.}

To close out this introduction, we discuss the quite limited progress on Greenberg-Shalom's question: First, Margulis proved that non-arithmetic lattices have discrete commensurators, thereby resolving Greenberg-Shalom's question for lattices. At roughly the same time and apparently unaware of Margulis' work, Greenberg proved Hypothesis \ref{q:shalom} for finitely generated subgroups of $G=\SL(2,\bbR)$ \cite{greenberg-comm}. Building on work by Leininger-Long-Reid \cite{llr-comm}, Mj proved Hypothesis \ref{q:shalom} for finitely generated subgroups of $\SL(2,\bbC)$ \cite{mj-comm}.  For all other cases of $G$, as well as general infinitely generated subgroups of the above, Hypothesis \ref{q:shalom} is open. Koberda-Mj proved the hypothesis for normal subgroups $\Gamma$ of arithmetic lattices in rank 1 with positive first Betti number \cite{koberda-mj,koberda-mj2}. Fisher-Mj-Van Limbeek proved the hypothesis for all normal subgroups of lattices \cite{f-mj-vl}.  Greenberg (for $\SO(n,1)$) and Mj (in general) also proved that any group
satisfying the conditions of the hypothesis has full limit set \cite{greenberg-comm, mj-comm}.

\subsection*{Acknowledgments} The second author thanks Yves Benoist for a very useful conversation in Zurich in January $2019$.  The authors
thank Michael Larsen and Sandeep Varma for useful conversations about local fields of characteristic $p$.  They also thank Simon Machado for
conversations about approximate subgroups and commensurated approximate subgroups as well as for reminding them of the example of Burger and Mozes.
  We also thank Marco Linton, Sam Mellick and Alan Reid for corrections to and comments on a preliminary version
of this paper. Finally the authors thank Alan Reid for encouragement.

David Fisher is supported by NSF DMS-2246556.  Mahan Mj is supported by the Department of Atomic Energy, Government of India, under project no.12-R\&D-TFR-5.01-0500;
and in part by a DST JC Bose Fellowship,  and an endowment from the Infosys Foundation. Wouter van Limbeek is supported by NSF DMS-2203867.

\section{Notation, standing assumptions, and definitions of irreducibility}\label{sec:assume}

In this section, we will fix some standing assumptions that will be in place for the rest of this article. Subsequent sections may have further assumptions that will be detailed at the beginning of each section.
\begin{sassume}\label{assumptions}\mbox{}
\begin{itemize}
	\item $I$ is a finite set indexing local fields $k_i, i \in I$, of characteristic zero.
	\item { $I$ also indexes a set of groups $\bbG_i$ which are} connected, absolutely simple, isotropic, adjoint algebraic groups over $k_i$. We set $G_i:=\bbG_i(k_i)$ and $G=\prod_{i\in I} G_i$. When not otherwise specified, we equip $G$ with the analytic topology.
\end{itemize}
\end{sassume}

\begin{rem} Much (but not all) of our discussion also applies when $k_i$ are local fields of positive characteristic, and the Greenberg-Shalom problem can be formulated, and is open and interesting, in that setting as well.\end{rem}

We now make two definitions:

\begin{dfn}
A subgroup $\Theta<G$ is \emph{almost dense} if its analytic closure has finite index in $G$.
\end{dfn}

\begin{dfn}
A subgroup $\Gamma<G$ will be called a \emph{standard commensurated subgroup} if it has unbounded projection to each
factor and almost dense commensurator $\Delta<G$.
\end{dfn}

\begin{notation} Let $G$ be as above. For $i\in I$, we denote by $\pr_i:G\to G_i$ the canonical projection. For $J\subseteq I$, we set $G_J:=\prod_{j\in J} G_j$, and $\pr_J: G\to G_J$ denotes the canonical projection.
\end{notation}

We introduce the following the following irreducibility assumption:

\begin{dfn}\label{dfn:s-irred} Let $G$ be as in the above standing assumptions and let $L\subseteq G$ be a subgroup. We say $L$ is \emph{strongly irreducible} if for every proper subset $J\subsetneq I$, the projection $\pr_J(L)\subseteq G_J$ is almost dense.\end{dfn}

From the above definition, it is not at all obvious that questions involving general discrete (or closed) groups can be reduced to strongly irreducible discrete (or closed) groups, i.e. if $\Theta\subseteq G$ is a discrete subgroup, it is not clear there exists $J\subseteq I$ such that $\pr_J(\Theta)\subseteq G_J$ is discrete and strongly irreducible. {  We now introduce a notion of \emph{irreducibility} for which reductions are easier.  In Appendix \ref{subsec:irr}, we will prove these are equivalent.}


To make our second notion of irreducibility precise we need to pass to the case where the only fields allowed
are $\bbQ_p$ and $\bbR$ and not finite extensions of those fields.  Many readers may already be assuming that and those that
are need not pay too much attention to the next two paragraphs on a first reading.

Let $G=\prod_i G_i$ be as in Standing Assumptions \ref{assumptions}, and let $\Gamma\subseteq G$ be a discrete subgroup. We want to define what it means for $\Gamma$ to be \emph{irreducible} in the product $G=\prod_i G_i$. There are two obvious ways one might reduce the study of $\Gamma$ to a simpler scenario, namely by passing to either a quotient of $G$ or a  subgroup of $G$. We can pass to a quotient precisely when the projection of $\Gamma$ to some proper subset of factors is discrete. To see when one may pass to a subgroup is more involved, and we need to introduce some notation.

We set $\bbQ_\infty:=\bbR$ and if $p$ is a finite prime or $\infty$, we denote by $I_p\subseteq I$ the subset of $\bbQ_p$-analytic factors of $G$ (so $I_\infty$ denotes the collection of archimedean factors), and we write $G_p:=\prod_{i\in I_p} G_i$ for their product. We restrict scalars on all factors to $\bbQ_p$, and define the $\bbQ_p$-algebraic group $R_p \bbG:= \prod_{i \in I_p} R_{k_i / \bbQ_p}  \bbG_i$. We write $R_p G:= (R_p \bbG)(\bbQ_p)$. Of course we have $R_p G \cong G_p$ as $\bbQ_p$-analytic groups, but $R_p G$ comes equipped with a Zariski topology as a $\bbQ_p$-algebraic group. We will refer to this as the $\bbQ_p$\emph{-Zariski topology} on $G_p$. Note that given a discrete group $\Gamma\subseteq G$, we could first project to any subproduct $G_J$ such that $\pr_J(\Gamma)$ is discrete, and then pass to the subgroup $H:=\prod_p \overline{\pr_{J_p}(\Gamma)}^{\bbQ_p}$, where the closure is taken with respect to the $\bbQ_p$-Zariski topology on $R_p G$. This motivates the following definition:

	\begin{dfn}\label{dfn:irred} A subgroup $\Theta$ is \emph{irreducible} in $G$ if for every proper subset of factors $J\subsetneq I$, the projection of $\Theta$ to $G_J=\prod_{j\in J} G_j$ is not discrete, and for every $p$ (a finite prime or $\infty$), the projection of $\Theta$ to $G_p$ is $\bbQ_p$-Zariski dense.
	\end{dfn}

Most irreducible groups we consider will be discrete, but we do not assume this for the purposes of the above definition.	
\begin{rem}\label{rmk:onefactorirred} If $G$ is simple, then irreducibility is equivalent to $\bbQ_p$-Zariski density. \end{rem}

{
\begin{lemma}
\label{rmk:irred->unbdd}
Any discrete irreducible subgroup $\Theta\subseteq G$ has unbounded projections to every simple factor $G_i$.  Moreover the set of elements projected
to any bounded open set in $G/G_i$ is already unbounded in $G_i$.
\end{lemma}

\begin{proof}
 If $G$ itself is simple, this follows from Zariski density of $\Theta$. If there is more than one factor, then for any factor $G_i$, the image of $\Theta$ in $G/G_i$ is indiscrete. For any bounded open set $U\subseteq G/G_i$, the set $\Theta_U$ of elements of $\Theta$ with image in $U$ is infinite and has discrete projection to $G_i$, and therefore is unbounded in $G_i$.\end{proof}}

{
To clarify the role of having two definitions, note that
irreducibility seems much easier to establish than strong irreducibility, since the former only requires $\bbQ_p$-Zariski density of projections to $R_p G$, whereas the latter requires almost density in the analytic topology of all subproduct projections. On the other hand, strong irreducibility is a much more powerful property to use in applications.  Since we can verify that the two notions are in fact equivalent, this apparent discrepancy is very useful.

\begin{prop}\label{prop:irred-vs-sirred}
A discrete subgroup is irreducible if and only if it is strongly irreducible. \end{prop}

This proposition will be proved in Appendix \ref{subsec:irr} where we will also prove the following Lemma which we use
frequently in applications and which provides the connection between irreducible groups and commensurators.

\begin{lemma}\label{lem:intersectionprojection}
Let $J$ index a collection of factors of $G$ such that the fields $k_i$ for $i \in J$ are all non-archimedean.
Let $K \subset \prod_J G_j$ be a maximal compact open.  Let $\Gamma_K = \Gamma \cap (G_{J^c} \times K)$.
Then $\pr_{J_c}(\Gamma_{K})$ is a discrete irreducible subgroup of ${G_{J^c}}$ with dense commensurator
$\pr_{J_c}(\Gamma)$.
\end{lemma}

We give here the proof for two factors one of which is a connected real Lie group, the other $p$-adic.  This is short, simple, and illustrates the key connection.
The more involved proof with many factors will be given in Appendix \ref{subsec:irr}.

\begin{proof}[Proof for two factors]
Let $G= G_1 \times G_2$ and assume $G_2$ is defined over a non-archimedean field and $G_1$ is a connected real Lie group.
By van Dantzig's theorem,  $G_2$ contains a compact open subgroup $K_2$ which is commensurated by $G_2$, see Proposition \ref{compactopencomm}. That proposition also
implies that $\Gamma_K=\Gamma \cap (G_1 \times K_2)$ is commensurated by $\Gamma$ simply by looking at the projection
to $G_2$.  Now $\Gamma_K$ projects discretely to $G_1$ since it is discrete in $G_1 \times K_2$ and the fiber
of the projection $G_1 \times K_2 \rightarrow G_1$ is compact. Now $\pr_1(\Gamma_K)$ is a discrete group
that is unbounded by Lemma \ref{rmk:irred->unbdd} and has almost dense commensurator $\pr_1(\Gamma)$. This
implies that the Lie algebra of the Zariski closure of $\pr_1(\Gamma_K)$ is invariant under $G$ and hence equals
 $G$, i.e. $\pr_1(\Gamma_K)$ is  Zariski dense, which is all we needed to show.
\end{proof}

}

\section{Irreducible groups are lattices}

\subsection{Irreducible groups in products of semisimple Lie groups} \label{sec:irred2latt} We can now give the first application of the Greenberg-Shalom hypothesis. Aside from its intrinsic interest, this result will provide an important connection to other applications as well.

\begin{prop}\label{prop-irred2latt} Assume the Greenberg-Shalom Hypothesis \ref{q:shalom}. Let $G$ be as in Standing assumptions \ref{assumptions} with at least two factors, and assume at least one factor is nonarchimedean. Then any discrete and irreducible subgroup $\Gamma\subseteq G$ is an irreducible lattice.\end{prop}

\begin{proof}

If $G$ has an archimedean factor, let $G_{0}$ denote the product of all archimedean factors and $G_{\text{na}}$ the product of all nonarchimedean factors. If there is no archimedean factor, let $G_{0}$ be one of the nonarchimedean factors and $G_{\text{na}}$ the product of the remaining factors. Let $\pr_{0}$ denote the projection to $G_{0}$.

Let $K_{\text{na}}\subseteq G_{\text{na}}^+$ be a maximal compact open subgroup. Let $\Gamma_{\text{na}}\subseteq \Gamma$ denote the subgroup that maps into $K_{\text{na}}$ under projection to $G_{\text{na}}$. Then $\pr_{0}(\Gamma_{\text{na}})\subseteq G_{0}$ is irreducible {(by Proposition~\ref{prop:irred-vs-sirred})}, discrete and has commensurator containing the dense subgroup $\pr_{0}(\Gamma)\subseteq G_{0}$. In particular $\pr_{0}(\Gamma_{\text{na}})$ is Zariski dense in $G_{0}$ (see Lemma \ref{lemma:zdense-comm}), and hence by the Greenberg-Shalom hypothesis, is a lattice. It then follows that $\Gamma\subseteq G$ is a lattice by Lemma~\ref{lem-venky}.
\end{proof}


{  We present below two consequences of Proposition~\ref{prop-irred2latt}. The proofs are postponed to Section~\ref{subsec-cors4} as in some places they will need some technical results on automorphism groups of trees.}
		
\begin{cor}
	\label{noirreduciblesurfacesinpadics}
	Assume the Greenberg-Shalom hypothesis and suppose $G$ has no archimedean factors and $|I|\geq 2$. Then there is no irreducible discrete surface group or irreducible discrete finitely generated free group in $G$. If in addition all simple factors of $G$ have rank one, then there is no discrete surface group in $G$.
\end{cor}

{ We note here that the first conclusion follows if one knows that an irreducible lattice in a product cannot be a free or surface
group. There are many simple ways to verify this. }

We call a surface group $\Lambda$ in $\PSL(2,\bbR)$ \emph{algebraic} if $\Lambda < \PSL(2, \overline{\bbQ})$.  By finite generation of surface
groups it follows that $\Lambda < \PSL(2, k)$ for some number field $k$.
{We assume that $k$ is minimal.} Let $\calO_k$ be the ring of integers of $k$.  Then
we have:

\begin{cor}
	\label{cor:algebraicsurfacegroups}
	Let $\Lambda$ be an algebraic surface group and assume the Greenberg-Shalom hypothesis.  Then $\Gamma := \Lambda \cap \PSL(2, \calO_k)$ is
	infinite, commensurated by $\Lambda$ and Zariski dense in $\PSL(2,k)$.
\end{cor}

{ 
\noindent {\bf Historical notes and questions:}\\
We refer the reader to Section~\ref{subsec-cors4} for proofs of the above two corollaries. However, a word is in order regarding the history of the
circle of questions that leads to Corollaries~\ref{noirreduciblesurfacesinpadics} and
~\ref{cor:algebraicsurfacegroups}.}

The existence of discrete irreducible surface subgroups in products of $p$-adic groups is a folklore problem that was widely discussed at MSRI in 2015. A variant for actions on products of trees was asked explicitly by Fisher-Larsen-Spatzier-Stover \cite{FLSS}. {We shall discuss this in Section~\ref{sec:flssimproved}.} In a conversation in January 2019, Yves Benoist pointed out to the second author that it was also unknown if there are irreducible free groups in products of real and $p$-adic Lie groups or irreducible surface groups in products of real Lie groups. {In fact, Benoist first informally posed this question as far back as 2003 or 2004. We thank Tsachik Gelander for first telling us the rough age of Benoist's question.}

Further, Long and Reid, using some ideas of Magnus \cite{Magnus}, constructed an explicit surface group in $\PGL(2, \mathbb{Q}_2) \times \PGL(2, \mathbb{Q}_3)$ generated by the two matrices: $$a=\begin{pmatrix} 3 &0 \\ 0 & \frac{1}{3}\end{pmatrix}, \qquad b=\begin{pmatrix} \frac{1}{8} & 9 \\ \frac{1}{32} & \frac{41}{4}\end{pmatrix}$$
\noindent This is an orbifold group with an index 4 subgroup that corresponds to a cover by a surface of genus two.  Extensive computations by Long-Reid, Agol and Brody suggest that this example is a discrete, irreducible subgroup of the product but according to Proposition \ref{prop-irred2latt}, this would give a negative answer to Greenberg-Shalom's Question \ref{q:shalom}.  It is worth comparing with the computer-assisted results of Kim-Koberda \cite{kim-koberda-notfree} that are relevant to another application below (see Section \ref{sec:lukk}), where groups were shown to be non-free by finding extremely long relators.  In fact, R~.Yaari recently showed that in that context, groups with arbitrarily long shortest relators necessarily occur \cite{Yaari}. If Hypothesis \ref{q:shalom} is correct, then the computational evidence above about the Long-Reid group only indicates that one needs extremely long words in $a$ and $b$ to find elements of integral trace.

In Corollaries~\ref{noirreduciblesurfacesinpadics} and \ref{cor:algebraicsurfacegroups}, we require at least one factor to be totally disconnected (i.e. a $p$-adic Lie group). As mentioned above, the second author learned of variants of this question where both factors are real Lie groups from Yves Benoist in January 2019.  The simplest and perhaps most intriguing case is:

\begin{question}[Benoist]
Is there a free group that acts properly discontinuously and irreducibly on $\bbH^2\times\bbH^2$?
\label{q:benoist}
\end{question}

We believe the requirement of a totally disconnected factor in Proposition \ref{prop-irred2latt} is an artefact of the method rather than a genuine difference. By analogy with Proposition \ref{prop-irred2latt} we propose the following strengthening of Benoist's question:

\begin{problem}
Let $G_1, G_2$ be real semisimple Lie groups with finite center and no compact factors, and suppose $\Delta\subseteq G_1\times G_2$ is a discrete and irreducible subgroup. Is $\Delta$ an arithmetic lattice? \label{prob:irred_arch}
\end{problem}

Just as in Problem \ref{prop-irred2latt}, irreducible groups $\Delta$ as in the above problem give rise to objects in one factor with large commensurator. But since $G_2$ is no longer totally disconnected, they are no longer groups: More precisely, consider $\Gamma:=\Delta\cap(G_1\times U)$ where $U\subseteq G_2$ is an open neighborhood of identity with compact closure. Note that $\Gamma$ is no longer a subgroup, but only an \emph{approximate subgroup}.  This approximate subgroup is commensurated by $\Delta$ in the sense of \cite[Definition 2.1.5]{Machado2}.
{  For completeness, we include the relevant definitions here.

\begin{dfn}\label{def-app}
Let $\Theta$ be a subset of a locally compact group $G$. We say that
$\Theta$
is an \emph{approximate subgroup} if
\begin{enumerate}
\item $\Theta$  is symmetric, i.e.\  $\Theta=\Theta^{-1}$,
\item $1 \in \Theta$, and
\item there exists a finite set $F \subset G$ such that
$\Theta^2 \subset F\Theta$. \\
\end{enumerate}
	\end{dfn}

	\begin{dfn} Let $X, Y$ be  subsets of a locally compact group $G$.
		We say that $X, Y$ are
\emph{commensurable}  if there  exists a finite set $F \subset G$ such that
	that
	$ X \subset  F Y \cap Y F$ and $Y \subset  F X \cap XF$. We say that
	$g \in G$
\emph{commensurates} $X$ if $gXg^{-1}$ and $X$ are commensurable.
	\end{dfn}

}

 By analogy with Proposition \ref{prop-irred2latt} we suggest the following variation of Greenberg-Shalom's Question \ref{q:shalom} as a way of approaching Problem \ref{prob:irred_arch}:

\begin{question} \label{q:approx}
Let $G$ be as in the Standing Assumptions \ref{assumptions} and let $\Gamma$ be a discrete, Zariski dense, approximate subgroup with almost dense commensurator. Is $\Gamma$ an approximate lattice?\end{question}

We remark here that for approximate subgroups of $G$, there is a well-defined notion of being an approximate lattice and of arithmeticity, and Machado (for real higher rank Lie groups) \cite{machado} and Hrushovski (in general) \cite{hrushovski} have simultaneously classified approximate lattices in products of simple algebraic groups defined over local fields: namely, for any approximate lattice $\Lambda\subseteq G$, we can decompose $G=G_1\times G_2$ such that $\Lambda$ is commensurable to a product of a lattice in $G_1$ and an arithmetic approximate lattice in $G_2$. Here, an \emph{arithmetic approximate lattice} in $G$ is either an arithmetic lattice or is obtained as the intersection of an irreducible lattice in $G\times H$ with $G\times U$, where $U$ is an open neighborhood of identity in $H$.  For a thorough account of the state of the art on approximate lattices, see Machado's recent preprint \cite{Machado2}.

\subsection{Groups generated by parabolic elements} \label{sec:lukk} For the remainder of this section, we discuss applications of rigidity of irreducible groups in semisimple Lie groups to problems that do not seemingly involve irreducible groups or commensurators. To state this problem, we let $q\in\bbC$ be a parameter, and we let $\Delta_q$ be the group generated by
	$$a=\begin{pmatrix} 1 & 0 \\ 1 & 1\end{pmatrix}, \qquad b_q=\begin{pmatrix} 1 & q \\ 0 & 1\end{pmatrix}.$$
The algebraic structure of $\Delta_q$, especially whether it is free or not, has been the subject of much work, starting with Sanov's 1947 theorem that $\Delta_4$ is free \cite{sanov}, and Brenner's theorem that $\Delta_q$ is free and discrete if $|q|>4$ \cite{brenner}, which implies $\Delta_q$ is free for transcendental values of $q$. Freeness of $\Delta_q$ for nonzero rational values of $q$ in $(-4,4)$ is a long-standing open problem first mentioned by Brenner and Hirsch. The question was really first studied in the work of Lyndon-Ullman, and later a negative answer has been conjectured by Kim-Koberda:
\begin{conj}[{Lyndon-Ullman, Kim-Koberda \cite{lyndon-ullman, kim-koberda-notfree}}]
For nonzero rational values of $q$ in $(-4,4)$, the group $\Delta_q$ is not free.
\label{conj:lu}
\end{conj}
Kim-Koberda have proved this for numerators up to 24. Further experimental results are obtained by Detinko-Flannery-Hulpke \cite{lu-experiments}. For $|q|<4$ (not necessarily rational), Knapp has proved that $\Delta_q$ is discrete if and only if $1-q/2=\cos((n-2)\pi/n)$ for some integer $n\geq 3$ \cite{knapp-lu}. Indeed, note that the element $ab_q^{-1}$ has trace $|2-q|<2$, hence it is elliptic. When $q$ is a rational non-integer, this must be an elliptic of infinite order. The case that $\Delta_q$ is discrete has been further studied by Agol, showing in particular that $\Delta_2$ and $\Delta_3$ are not free, see \cite{agolwriteup}. Knapp's theorem immediately implies $\Delta_q$ is indiscrete for nonintegral rational values with $|q|<4$. A complete description of $\Delta_q$ up to finite index, and in particular a positive answer to the above conjecture, follows from Greenberg-Shalom's problem:

\begin{thm} \label{thm:lukk}
Assume the Greenberg-Shalom Hypothesis \ref{q:shalom}. Then for every non-integral rational $q=r/s\in (-4,4)$, the group $\Delta_q$ has finite index in $\SL(2,\bbZ[1/s])$. In particular, Conjecture \ref{conj:lu} is true.
\end{thm}

\noindent In this context it is particularly worth mentioning that the work of Detinko-Flannery-Hulpke \cite{lu-experiments} proves that for certain {values of $q$, $\Delta_q$ is finite index in $\SL(2,\bbZ[1/s])$, see
\cite[Table 1]{lu-experiments}}.  There are currently no known values of $q$ for which the (conditional) conclusion of Theorem \ref{thm:lukk} fails to hold.

\begin{proof} If $q=r/s$ is written in lowest terms and is not integral, then $\Delta_q$ is a discrete subgroup of $G:=\SL(2,\bbR)\times \prod_{p\mid s} \SL(2,\bbQ_p)$. Let $T$ be a minimal subset of factors such that $\Delta_q$ has discrete projection to $G_T$. Note that $T$ consists of at least two factors: The projection of $\Delta_q$ to $\SL(2,\bbR)$ is not discrete by the above-mentioned work of Knapp \cite{knapp-lu}, and its projection to $\SL(2,\bbQ_p)$ is not discrete because $a\in \SL(2,\bbZ_p)$ has infinite order. Further the projections are Zariski dense since their Zariski closures cannot be solvable, and $\SL(2,\bbR)$ and $\SL(2,\bbQ_p)$ do not contain proper Zariski closed non-solvable subgroups. Therefore $\pr_T(\Delta_q)\subseteq G_T$ is irreducible, and by Proposition \ref{prop-irred2latt}, $\pr_T(\Delta_q)$ is a lattice in $G_T$. Note that $T$ must contain $\infty$, because $a$ generates an indiscrete subgroup of $\prod_{p\mid s} \SL(2,\bbQ_p)$. So $\Delta_q\subseteq \SL(2,\bbZ[1/s])$ is $T$-arithmetic, and is therefore commensurable with $\SL(2,\bbZ[1/t])$, where $t$ denotes the product of the finite primes in $T$.

To see that $t=s$, it suffices to show that $\Delta_q$ has unbounded projection to $\SL(2,\bbQ_p)$ for all $p\mid s$. Indeed, $ab$ has trace $2-q = 2-r/s$, and is therefore not elliptic in $\SL(2,\bbQ_p)$. In particular, $ab$ generates an unbounded subgroup. \end{proof}
In fact, this strategy can handle not just  rational, but all algebraic values of $q$:
\begin{thm}
\label{thm:genLU}
Assume the Greenberg-Shalom Hypothesis \ref{q:shalom}. Let $q\in\overline{\bbQ}$ be an algebraic number that is not an algebraic integer, and let $\Delta_q$ denote the corresponding Lyndon-Ullman group.

Then $\Delta_q$ is free if and only if there is a Galois automorphism $\sigma \in \Gal(\bar{\bbQ}/\bbQ)$ such that $\Delta_{\sigma(q)}$ is free and discrete.
\end{thm}
\begin{proof}[Proof of Theorem \ref{thm:genLU}] {We first note that for any Galois automorphism $\sigma \in \Gal(\bar{\bbQ}/\bbQ)$, $\Delta_{\sigma(q)}$ is free if and only if $\Delta_q$ is  free. Hence, assume that for any Galois automorphism $\sigma \in \Gal(\bar{\bbQ}/\bbQ)$, $\Delta_{\sigma(q)}$ is indiscrete. We want to show that $\Delta_q$ is not free.}
	
	The proof is similar to the previous one. Let $k=\bbQ(q)$. Since $q\notin \calO_k$,  there exists a valuation $\nu$ of $k$ such that $|q|_\nu>1$. Let $p$ be the prime such that $k_\nu$ is a finite extension of $\bbQ_p$ and consider $\Delta_q\subseteq \PSL(2,k_\nu)$. Note that $\Delta_q$ is not discrete since $a\in \Delta_q\cap \PSL(2,\calO_\nu)$ has infinite order. Therefore to show $\Delta_q$ is almost-dense in a semisimple $p$-adic Lie group, it suffices to show its $\bbQ_p$-Zariski closure is semisimple. To see this, it suffices to show the Lie algebra $\frg_q$ of the $\bbQ_p$-Zariski closure of $\Delta_q$ is semisimple.

Note that $\Delta_q$ is $k_\nu$-Zariski dense in $\PSL(2,k_\nu)$ because it is not virtually solvable. This implies the $k_\nu$-span of $\frg_q$ is all of $\mathfrak{sl}(2,k_\nu)$ (because it is a $\Delta_q$-invariant Lie subalgebra), and therefore $\frg_q$ must itself be semisimple: The $k_\nu$-span of its solvable radical would be an ideal in the $k_\nu$-span of $\frg_q$.

{By assumption,  for any Galois automorphism $\sigma \in \Gal(\bar{\bbQ}/\bbQ)$ $\Delta_{\sigma(q)}$ is indiscrete. Hence, by} the Greenberg-Shalom hypothesis, $\Delta_q$ is an arithmetic lattice in a semisimple $p$-adic Lie group $H$. Note that $H$ has at least two factors because $\Delta_q$ is not discrete in its $\bbQ_p$-Zariski closure inside $\PSL(2,k_\nu)$. In particular, $\Delta_q$ is a higher rank arithmetic lattice and hence not free.\end{proof}

\subsection{Arithmetic of hyperbolic 3-manifolds} \label{sec:arith-3dim}

In this section we let $G=\PO(3,1)=\Isom(\mathbb{H}^3)$ and discuss an application to hyperbolic 3-manifolds, i.e. to manifolds of the form $K \backslash G/\Lambda$ where $\Lambda$ is a lattice.  It follows from
{\cite[Corollary 3.2.4]{Maclachlan-Reid}} that there is a number field $k$ such that $\Lambda < G(k)$ {(after conjugation by an element of $G$ if necessary).
More generally, by work of Selberg, Calabi, Raghunathan, and Garland \cite{Selberg, Calabi, RaghunathanRigid, Garland} for $G$ a rank one Lie
group, not locally isomorphic to $\SL(2,\R)$, the same statement
continues to hold for lattices in $G$. Unless $G$ is of the form $\SO(n,1)$ stronger statements are known.  In higher rank Margulis famously showed all lattices are arithmetic, a conclusion confirmed also for lattices in $\Sp(n,1)$ and $F_4^{-20}$ by Corlette and Gromov-Schoen \cite{margulis-icm, corlette,gromovschoen}.  Very recently, Esnault-Groechenig have shown that all lattices in $\SU(n,1)$  are in fact conjugate into the integer points of some number field \cite{esnault-groechenig}.  For an account of why the results of that paper imply the desired result the reader can see \cite[Theorem 1.(3)]{BFMS2} or \cite[Theorem 4.3]{BU}.}

It is a natural and reasonably well-known question to ask whether $\Lambda$ necessarily contains any integral matrices when $G=\SO(n,1)$.  We let $\mathcal{O}_{k}$ denote the ring of $k$-integers.

For completeness, we first furnish a proof of the following known statement for non-uniform lattices.
 We thank Alan Reid for drawing this fact and its proof to our attention.
We give a statement for non-uniform lattices in all rank one Lie groups, even though except in the case of $\SO(n,1)$, stronger results are discussed above.

{ 	For $G$  a rank one
	Lie group not locally isomorphic to $\SL(2,\R)$,  any lattice $\Lambda<G$ is conjugate into $G(k)$ for some number field $k$
	as noted above. We assume below that $k$ is a minimal extension of $\Q$.}

\begin{thm}\label{thm-nonu}
	Let $G$ be a rank one Lie group not locally isomorphic to $\SL(2,\R)$ and $\Lambda<G$ a non-uniform lattice. Let $k$ be as above, and
	$\mathcal{O}_{k}$ its ring of integers. Then {the} subgroup $\Gamma := \Lambda \cap G(\mathcal{O}_{k})$
	is infinite, commensurated by $\Lambda$, and Zariski dense in $G(k)$.
\end{thm}

\begin{proof} { 	Since $\Lambda$ is finitely generated, the set of places $S$ of $k$ such that there exists a matrix entry of a generator of $\Lambda$ with negative valuation is finite. It follows immediately that $\Lambda$ commensurates $\Gamma$. Since the commensurator of $\Gamma$ is Zariski dense, as soon as $\Gamma$ is infinite, $\Gamma$ itself is Zariski dense (see Lemma \ref{lemma:zdense-comm} and Remark \ref{rem:nrfield}).}
	
	It thus suffices to show that $\Gamma$ is infinite.
	By results of Garland-Raghunathan \cite{GarRag}, we know that for some choice of parabolic $P$, $\Gamma$ intersects the unipotent radical $U$ of $P$ in a lattice.  Therefore
	$\Gamma \cap Z(U)$ is also non-trivial where $Z(U)$ is the center of $U$.  If $\gamma \in \Z(U)(k)$ then it is elementary that for some $n$, we have $\gamma^n \in Z(U)(\mathcal{O}_k)$ since on the center of a unipotent group, multiplication is just addition of matrix coefficients.
	Hence, $\Gamma$ is infinite.
\end{proof}

As an application of our methods, we prove the following more general,
but conditional statement that holds when $\Lambda$ is uniform and $G=\SO(3,1)$:

\begin{thm}
\label{thm:3manifold}  Assume the Greenberg-Shalom Hypothesis \ref{q:shalom}. For any finite volume hyperbolic 3-manifold $M=\bbH^3/\Lambda$ and $k, \mathcal{O}_{k}$ as above, the subgroup $\Gamma := \Lambda \cap G(\mathcal{O}_{k})$
is infinite, commensurated by $\Lambda$, and Zariski dense in $G(k)$.
\end{thm}

\begin{proof}
As discussed above, along with the number field $k$, there exists  an algebraic group $\bbG$ defined over $k$ such that $\Lambda\subseteq \bbG(k)$. Here $\Lambda$ is a lattice in $\bbG(\bbR)=\PO(3,1)$, so $\bbG$ is a form of $\PO(4)(\mathbb{C})$. As in the proof of Theorem~\ref{thm-nonu},
the set of places $S$ of $k$ such that there exists a matrix entry of a generator of $\Lambda$ with valuation $>1$, is finite. It follows immediately that $\Lambda$ commensurates $\Gamma$.
 It remains to prove that $\Gamma$ is infinite.

At any place $s\notin S$, we have $\Lambda\subseteq \bbG(\mathcal{O}_s)$. In particular, if $S=\varnothing$, then $\Gamma=\Lambda$ and we are done. Now assume $S\neq \varnothing$. If $\Lambda\subseteq G_S$ is indiscrete, then $\Gamma=\Lambda\cap \bbG(\mathcal{O}_S)$ is infinite and we are done.

Henceforth assume that $\Lambda\subseteq G_S$ is discrete. Since $\PO(4)$ is of type $D_2$ and any split group of type $D_2$ is a product of groups of type $A_1$ (cf.\ \cite{stack}), the group $G_S:=\prod_{s\in S} \bbG(k_s)$ splits as a product of rank one factors, say $G_S = \prod_{i\in I} G_i$. Choose a minimal subset $J\subseteq I$ (possibly $J=I$) such that $\pr_J(\Lambda)\subseteq G_J$ is discrete. We note that $J$ consists of at least two factors. Indeed, since each $G_i$ has rank 1, its Bruhat-Tits building is a tree, and since $\Lambda$ is not virtually free, it cannot act properly discontinuously on a bounded degree tree.

Since $\Lambda\subseteq\bbG(k)$, its $\bbR$-Zariski closure in $\bbG(\bbR)$ is defined over $k$ (see e.g. \cite[Proposition 3.1.8]{zimmer-book}). Since $\Lambda\subseteq\bbG(\bbR)$ is a lattice, it is $\bbR$-Zariski dense in $G(\mathbb{R})$. Combining these two observations, we see that $\Lambda$ is $k$-Zariski dense in $\bbG(k)$. Restricting scalars from $k$ to $\bbQ$, we conclude that $\Lambda$ is $\bbQ$-Zariski dense in $(R_{k/\bbQ}\bbG)(\bbQ)$.

{Let $J_p \subset J$ denote the indices corresponding to
	the $p$-adic factors of $G_J$ and $G_{J_p}$ denote the product of the $p$-adic factors of $G_J$. Then, by restriction of scalars as in the discussion following Definition \ref{dfn:s-irred}, $G_{J_p}$ is a $\bbQ_p$-algebraic group.}

We claim that $\Lambda$ is irreducible in $G_J$. By minimality of $J$, the projection of $\Lambda$ to any proper subset of factors is indiscrete, so it remains to show that $\Lambda$ is Zariski dense in $G_{J_p}(\bbQ_p)$ for every prime $p$.

 Now write $S_p$ for the set of $p$-adic places of $S$. We have
	$$G_{S_p} = \prod_{s\in S_p} \bbG(k_s) = \prod_{s\in S_p} (R_{k/ \bbQ} \bbG)(\bbQ_p).$$
We regard $G_{S_p}$ as a $\bbQ_p$-algebraic group using the product structure given by the right-hand side. Since the diagonal embedding
	$$R_{k/\bbQ}\bbG\hookrightarrow \prod_{s\in S_p} R_{k/\bbQ} \bbG$$
is defined over $\bbQ$ and $\Lambda$ is $\bbQ$-Zariski dense in $(R_{k/\bbQ}\bbG)(\bbQ)$, the $\bbQ$-Zariski closure of $\Lambda$ in $\prod_{s\in S_p} (R_{k/\bbQ} \bbG)(\bbQ)$ is exactly the diagonal. So the $\bbQ_p$-Zariski closure of $\Lambda$ in $G_{S_p}$ contains the diagonally embedded copy $\bbG(k)$. However, since $\bbG$ is adjoint, it satisfies weak approximation, so $\bbG(k)$ is (analytically) dense in $G_{S_p}$
{\cite[Corollary 5.8]{sansuc}}. Hence the $\bbQ_p$-Zariski closure of $\Lambda$ in $G_{S_p}$ (which is of course analytically closed) contains a dense set and therefore must be all of $G_{S_p}$. This shows that $\Lambda$ is $\bbQ_p$-Zariski dense in $G_{S_p}$.

Since every factor of $G_{S_p}$ is defined over $\bbQ_p$, the projection $G_{S_p}\to G_{J_p}$ is also defined over $\bbQ_p$. It follows that the image of $\Lambda$ in $G_{J_p}$ is also $\bbQ_p$-Zariski dense.\end{proof}

\section{Coherence}\label{sec:cohere} {The main result of this section is Theorem~\ref{thm-cohere}, deducing coherence of certain
$S-$arithmetic lattices from the Greenberg-Shalom Hypothesis~\ref{q:shalom}.}
Recall the following definition of coherent groups:
\begin{dfn}\label{def-cohere}
A finitely presented group $\Gamma$ is \emph{coherent} if any finitely generated subgroup of $\Gamma$ is  finitely presented.
\end{dfn}
Free groups and surface groups are coherent, as are abelian and, more generally, polycyclic groups. Scott has shown 3-manifold groups are coherent \cite{scott-cohere}, and since then, coherence has emerged as one of their salient properties. On the other hand, $F_2\times F_2$ is incoherent.
{This was first observed by Baumslag, Boone, and Neumann in \cite{bbn} and this line of inquiry was developed by Stallings \cite{stallings}, cf.\ \cite[Section 9.d]{wise-cohere}.}  Baumslag-Roseblade proved that a subgroup of $F_2\times F_2$ is finitely presented if and only if it is a finite extension of a finite product of finite rank free groups, and that there are many finitely generated subgroups that are not of this form \cite{baumslag-roseblade}.

See \cite{wise-cohere} for a survey of coherent groups, as well as many open problems. The class of coherent groups is not well-understood: For example, Serre famously asked whether SL($2,\bbZ[1/p])$ is coherent \cite{serre-cohere}. Below we give a positive answer assuming the Greenberg-Shalom hypothesis. Note however that $\SL(4,\bbZ)$ is incoherent as it contains the incoherent group $F_2\times F_2$.  Serre also asked if $\SL(3,\bbZ)$ is coherent, and this also remains open.

The relevance of the Greenberg-Shalom hypothesis to coherence of lattices is that by Proposition \ref{prop-irred2latt}, irreducible lattices in products do not have many subgroups; more precisely either a subgroup is not irreducible or it is also a lattice and therefore has finite index. In the following case, we have sufficient control over reducible groups that we can conclude they are finitely presented:

\begin{thm}\label{thm-cohere}
	Assume the Greenberg-Shalom Hypothesis \ref{q:shalom}. Let $S$ be a (nonempty) finite set of places of $\bbQ$ and either assume all places are finite or that there are two places, one infinite and one finite, in $S$. Set $G_s:=\PGL(2,\bbQ_s)$ and $G:= G_S$. Then any $S$-arithmetic lattice $\Lambda\subseteq G$ is coherent.
\end{thm}

\begin{proof} \mbox{}
Let $\Gamma\subseteq\Lambda$ be a finitely generated infinite subgroup. Let $H$ denote the $\Q$-Zariski closure of $\Gamma$.  If $H$ is a proper subgroup then {it must fix some point on $G/P$ since it cannot fix a proper factor. Hence, it is solvable.}
In the case when we have one real and one finite place, this implies that $\Gamma$ is either virtually abelian or virtually contained in a Baumslag-Solitar
group $BS(1,p)$.  In the case where there are no real places, $H$ solvable implies $H$ abelian, since all solvable subgroups of the lattice are virtually abelian. In all cases, $\Gamma$ is finitely presented and we are done.  So we can assume $\Gamma$ is $\Q$-Zariski dense and so all projections
are Zariski dense.

Choose a minimal subset $T\subseteq S$ (possibly $T=S$) such that $\pr_T(\Gamma)\subseteq G_T$ is discrete. Then  $\pr_T(\Gamma)$ is irreducible in $G_T$. If $|T|\geq 2$, then by Proposition \ref{prop-irred2latt}, $\Gamma$ is an irreducible lattice in $G_T$ and hence is finitely presented, and the proof is complete.

It remains to consider the case that $\Gamma$ projects discretely to a factor $G_t$, which is either $\PGL(2,\bbR)$ or $\PGL(2,\bbQ_p)$. In either case, all discrete, finitely generated subgroups are finitely presented: In $\PGL(2,\bbR)$, any such subgroup is a surface group or is virtually free. In $\PGL(2,\bbQ_p)$, every such group is virtually free. \end{proof}

\begin{rem}
It is clear from the proof that something more general can be proven by the same argument.  For example, we can consider number fields other than $\Q$ if we assume
all places are finite, {as in this case, lattices are cocompact and solvable subgroups of $\Lambda$ are abelian}.  A similar argument might work for (some) other number fields if we allow one infinite place and one finite place, {provided the solvable subgroups of $\Lambda$ are coherent.  Thus, the restrictions really stem from the fact that}
 it is known that $\SL(2, \Z[1/n])$ is not coherent if $n$ is composite \cite{serre-cohere}.  It  also clearly follows from the Greenberg-Shalom hypothesis that all finitely generated Zariski dense subgroups of  $\SL(2, \bbQ)$ are finitely presented.
\end{rem}

This also has implications for a fundamental question about coherence, to our knowledge first explicitly posed by Wise \cite[Problem 9.16]{wise-cohere}, namely whether coherence is a geometric (i.e. quasi-isometry invariant) property? We note that Wise already explicitly hedges against this.

\begin{cor}
Assume the Greenberg-Shalom hypothesis. Then coherence is not a quasi-isometry invariant.
\end{cor}
\begin{proof}
Let $p,q$ be distinct primes and consider lattices in $G:=\PGL(2,\bbQ_p)\times \PGL(2,\bbQ_q)$. All such lattices are cocompact and hence quasi-isometric. By Theorem \ref{thm-cohere}, the irreducible lattices in $G$ are coherent. On the other hand,  any lattice in $\PGL(2,\bbQ_p)$ or $\PGL(2,\bbQ_q)$ is virtually free (of rank $>1$), so $G$ admits reducible lattices that are products $F_m\times F_n$ of free groups (with $m,n>1$). These are incoherent by the previously mentioned result of Baumslag-Roseblade \cite{baumslag-roseblade}.
\end{proof}

\section{Margulis-Zimmer conjecture}\label{sec:mz} A major motivation for Greenberg-Shalom's question is the following conjecture advertised by Margulis-Zimmer in the late `70s, seeking to classify commensurated subgroups of higher rank lattices $\Lambda$. Here we say $\Gamma\subseteq\Lambda$ is \emph{commensurated} if $\Lambda\subseteq \text{Comm}_G(\Gamma)$.
{The main result of this section is Theorem~\ref{thm-mz}, deducing the Margulis-Zimmer Conjecture~\ref{conj:mz} below for certain
	$S-$arithmetic lattices from the Greenberg-Shalom Hypothesis~\ref{q:shalom}.}

\begin{conj}[{Margulis-Zimmer, see \cite{shalom-willis}}] Let $\G$ be a semisimple algebraic group defined over  a number field $k$ and $S$ a finite set of valuations of $k$. Assume $\G$ has higher $S$-rank. Assume $\Lambda=\G(\mathcal{O}_S)$ is an $S$-arithmetic lattice in $\G$, then any commensurated subgroup of $\Lambda$ is either finite or $S'$-arithmetic for some $S'\subseteq S$.
\label{conj:mz}
\end{conj}

Here, the $S$\emph{-rank} of $\G$ is the sum of the $k_{\nu}$-ranks over all valuations $\nu\in S$, and $\G$ is said to have \emph{higher }$S$\emph{-rank} if its $S$-rank is at least 2.

\begin{rem} \mbox{}\label{rmk:mz}
	\begin{enumerate}[(1)]
		\item For example, if $\Gamma$ is a commensurated subgroup of $\SL(n,\bbZ[1/p])$ (where $n\geq 2$), then $\Gamma$ is predicted to be either finite, finite-index, or commensurable to $\SL(n,\bbZ)$.
		\item Conjecture \ref{conj:mz} is motivated by and strengthens Margulis' Normal Subgroup Theorem.
		\item Question \ref{q:shalom} and Conjecture \ref{conj:mz} are closely related: For example, if $\Gamma\subseteq \G(\mathcal{O}_S)$ is a commensurated subgroup and $\Gamma\cap \G(\mathcal{O})$ is infinite, then a positive answer to Question \ref{q:shalom} proves $\Gamma$ intersects $\G(\mathcal{O})$ in a lattice. $S'$-arithmeticity of $\Gamma$ then follows from Venkataramana's result (see \cite[Proposition 2.3]{lubotzky-zimmer}) that the only intermediate subgroups between $\G(\mathcal{O})$ and $\G(\mathcal{O}_S)$ are $S'$-arithmetic for some $S'\subseteq S$, which proves Conjecture \ref{conj:mz} for such groups.
	\end{enumerate}
\end{rem}

Using well-chosen unipotent generating sets, Venkataramana has proven the Margulis-Zimmer Conjecture for arithmetic lattices $\Gamma=\G(\bbZ)$ in simple groups defined over $\bbQ$ \cite{venkataramana-mz}. Shalom-Willis have proved Conjecture \ref{conj:mz} in more instances, including the first that are not simple \cite{shalom-willis}. Their proof crucially relies on fine arithmetic properties for lattices in these groups, namely bounded generation by unipotents. As it is now known that bounded generation is not a common property for higher rank lattices \cite{CRRZ}, different approaches are needed. The results in \cite{f-mj-vl} yield partial results on this conjecture and the first that do not depend on unipotent elements, see Corollary $1.5$ in that paper.  This summarizes all prior work on Conjecture \ref{conj:mz}.

We now deduce a special case of Conjecture \ref{conj:mz} assuming the Greenberg-Shalom hypothesis.  Namely we deduce the case where $G$ has at least two factors,
at least one of which is non-archimedean. The deduction of this case of the Margulis-Zimmer conjecture from the Greenberg-Shalom hypothesis is quite simple modulo arguments we have already made. The existence of at least two factors is essential to the argument, but the requirement of a non-archimedean one can be removed if the Greenberg-Shalom hypothesis is strengthened to cover approximate groups (see Question \ref{q:approx}), but we do not pursue this here.

Let $G = \prod_I G_i$ be as in \ref{assumptions}. {We denote} the rank of $G$ by $\rk(G) = \sum_{i \in I} k_{i}$-$\rk(G_i(k_i))$ and assume that $\rk(G)\geq 2$.  Let $G_{\text{na}}$ denote the product
of $G_i$'s over all non-archimedean factors. We establish the following case of
Conjecture \ref{conj:mz}.

\begin{thm}\label{thm-mz} Assume the Greenberg-Shalom hypothesis and suppose $G_{\text{na}}$ is non-trivial and $|I|\geq 2$.
Let $\Lambda$ denote an irreducible lattice in $G$, and $\Gamma \subset \Lambda$ be an infinite subgroup commensurated by $\Lambda$. Then there exists $J \subset I$ such that
$\pr_J(\Gamma)$ (isomorphic to $\Gamma$ under $\pr_J$) is a lattice in $G_J$.
\end{thm}

\begin{proof}
If $\Gamma$ is irreducible, it is a lattice by Proposition \ref{prop-irred2latt}.
	
Assume therefore that $\Gamma$ is reducible. Hence there exists a minimal subset $J \subsetneq I$   such that $\pr_J(\Gamma)$ is discrete. Irreducibility of $\Lambda$ implies that  $\pr_J(\Lambda)$ is dense in $G_J$. Note also that  $\pr_J$ is injective on $\Lambda$, and hence on $\Gamma$. Thus, $\pr_J(\Gamma)\subset G_J$ is an infinite discrete subgroup commensurated by $\pr_J(\Lambda)$, where the latter is almost dense.  Then the Greenberg-Shalom hypothesis implies that $\pr_J(\Gamma)\subset \pr_J(G)$ is a lattice.
\end{proof}

One might want to reverse this implication, but it is not immediate.  If one assumes that $\Gamma<G$ as in Greenberg-Shalom's Question \ref{q:shalom} is contained in an arithmetic lattice $\bbG(\calO)\subseteq G$ (where $\calO$ is the ring of integers of a number field $k$), then it follows from an old argument of Borel that the commensurator $\Lambda:=\Comm_G(\Gamma)$ is contained in  $\bbG(k)$  see \cite{borel-comm} or \cite[Prop. 6.2.2]{zimmer-book}. (In \cite[Prop. 6.2.2]{zimmer-book} this is stated for arithmetic lattices, but the only property used is that $\Gamma$ is a Zariski dense subgroup of an arithmetic lattice). Then $\Lambda$ contains finitely generated subgroups that are dense in $G$ and contained in $S$-arithmetic lattices $\bbG(\calO_S)$, where $S$ is a finite set of places of $k$.  It is, however, not at all clear
in general how to force $\Lambda$ to contain an $S$-arithmetic lattice without already proving $\Gamma$ is a lattice. In this sense, the Greenberg-Shalom hypothesis is a
strengthening of a special case of the Margulis-Zimmer conjecture.

\section{Automorphism groups of trees}
\label{sec:flssimproved}

Going beyond semisimple Lie groups, one can study discrete irreducible subgroups of automorphism groups of products of trees. Their irreducible lattices are known to exhibit rigidity by the work of Burger-Mozes \cite{burger-mozes1,burger-mozes2}.
In this setting we have the following questions due to Fisher-Larsen-Spatzier-Stover:

\begin{question}[{Fisher-Larsen-Spatzier-Stover \cite{FLSS}}] \mbox{}Let $\Gamma$ be a surface group of genus $g\geq 2$.
	\begin{enumerate}[(1)]
		\item Does there exist $\rho: \Gamma \rightarrow \Aut(T_1 \times \cdots \times T_k)$ with discrete image?
		\item  Does there exist $\rho$ as in $(1)$ where $\rho$ takes values in a product $G_1 \times \cdots \times G_k$
		where each $G_i$ is a rank one simple algebraic group over a non-archimedean local field?
		\item Can $\Gamma$ be faithfully represented into $\PGL(2,K)$ for some global field $K$ of positive characteristic?
	\end{enumerate}
	\label{q:flss}
\end{question}

\noindent If there is an action as in $(1)$, then on a subgroup of finite index, $\rho$ is the diagonal embedding from homomorphisms $\rho_i: \Gamma \rightarrow \Aut(T_i)$. Question \ref{q:flss}$(2)$ is only implicit in \cite{FLSS}. In \cite{FLSS}, it is shown a positive answer to the third question gives a positive answer to the first and second. Questions \ref{q:flss}$(1)$ and $(2)$ are clearly related to Corollary \ref{noirreduciblesurfacesinpadics}.    We remark here that a negative answer to the analogue of Greenberg-Shalom's question for automorphism groups of trees is given by Burger and Mozes in \cite[Proposition 8.1]{burger-mozes0}, and that one can also construct irreducible subgroups in products of trees as done by e.g. the third author and Huang in \cite{huang-mj}. However, Question \ref{q:flss}$(1)$ is still open as is the corresponding question
for finitely generated free groups.

In this section we shall first prove an improvement of \cite[Theorem 15]{FLSS} and use it to prove Corollary \ref{noirreduciblesurfacesinpadics} and from that Corollary
\ref{cor:algebraicsurfacegroups}.
The improvement is minor and the main ideas are present in \cite{FLSS} but we need a stronger statement than given there and expect it to
be needed in other applications. {For the purposes of this section, we shall use the following terminology. Let $T$ denote a tree, and
	$\partial T$ its Gromov boundary. Then any $p \in \partial T$ will be referred to as a \emph{point at infinity} of $T$.}
\subsection{On a theorem from \cite{FLSS}}\label{subsec-flss}

\begin{thm}\label{thm:flssimproved}
Suppose that $\Lambda$ is a torsion free hyperbolic group that is not free. Let
\[
X = T_1 \times \cdots \times T_n
\]
be a product of finite-valence trees, set $G_i := \mathrm{Aut}(T_i)$, and $G := \prod G_i$. Let $\pr_i$ denote the projection of $G$ onto $G_i$. If $\rho : \Lambda \to G$ is a discrete and faithful representation, then there are at least two values of $i$ such that $\rho_i := \pr_i \circ \rho$ is faithful and has indiscrete image. Moreover, suppose $\rho_1, \dots, \rho_r$ are faithful representations and the other $\rho_i$ are not. Then the representation
\[
\rho_1 \times \cdots \times \rho_r : \Lambda \to G_1 \times \cdots \times G_r
\]
is discrete and faithful.  If we further assume $\rho_1 \times \cdots \times \rho_r$ is minimal for the property of having discrete image, then for all $i$, $\rho_i(\Lambda)$ does not fix a  point at infinity on $T_i$.
\end{thm}

\begin{proof}
The only new statement is the last one concerning no fixed points at infinity. Assume the action on one tree $T_1$ fixes a point at infinity, that all $\rho_i$ are faithful, that $r>1$ and that $(\rho_2 \times \cdots \times \rho_r)(\Lambda)$ is not discrete.  If we fix a point $\eta$ at infinity on $T_1$, we have a height, or Busemann, function $b_{\eta}: T_1 \rightarrow \bbZ$.  It is easy and standard that if $\rho_1(\Lambda)$ fixes $\eta$ and $x_0\in T_1$ is chosen such that $b_\eta(x_0)=0$,  then the map \begin{align*}
    B_\eta: &\Lambda \longrightarrow \bbZ\\
            &\lambda \longmapsto b_\eta(\lambda x_0)
\end{align*}
is a homomorphism.

Let $K = \ker(B_{\eta})$. Observe that since $G_1$ has no parabolic elements, any $\lambda \in K$ fixes a point in $T_1$.
Now fix a vertex $v$ in $T_2 \times \cdots \times T_r$ and let $\Delta$ be the stabilizer of $v$ in $\Lambda$ under the action defined by $\rho_2 \times \cdots \times \rho_r$.  Note $\Delta$ is non-trivial by hypothesis.

It suffices to show that $K \cap \Delta$ is nonempty and this is done almost exactly as in \cite{FLSS}.  Consider $x \in K$ and $y \in \Delta$.  As $K$ is normal in $\Lambda$,
the commutators $[x,y^n]=x(y^n x^{-n}y^{-n})$ belong to  $K$ for any $n\in\bbZ$.  It suffices to find values of $n$ such that the commutator also belongs to $\Delta$.  As in \cite{FLSS}, we see that
since $y$ fixes $v$, the points $y^n x^{-1} y^{-n} \cdot v =   y^n x^{-1} \cdot v$ all lie in a ball centered at $v$ of radius $d(v, x^{-1} \cdot v)$.  Since all trees are finite valence, this ball is a finite set and so there are $n_1 \neq n_2 \in \bbZ$ such that $y^{n_1} x^{-1} \cdot v  = y^{n_2} x^{-1}  \cdot v$  this implies $x y^{n_1-n_2} x^{-1} v =v$, so that $[x, y^{n_1-n_2}]\in \Delta$.  This contradicts discreteness of $(\rho_1 \times \cdots \times \rho_r)(\Lambda)$.
\end{proof}

{ 
\subsection{Proof of Corollaries~\ref{noirreduciblesurfacesinpadics} and
	~\ref{cor:algebraicsurfacegroups}}\label{subsec-cors4}
	
	With Theorem~\ref{thm:flssimproved} in place, we are finally in a position
	to prove Corollary~\ref{noirreduciblesurfacesinpadics}, whose proof we had postponed.}

\begin{proof}[Proof of Corollary~\ref{noirreduciblesurfacesinpadics}:]
	The first statement follows e.g. from the fact that a surface group $\Lambda$ can only be a lattice in a semisimple algebraic group if the group is locally isomorphic
	to $\PSL(2,\bbR)$ and that a finitely generated free group $F_k$ is only a lattice in a simple algebraic group of rank one over a non-archimedean field or  a group locally
	isomorphic to $\PSL(2,\bbR)$.
	
	For the second point, we have a surface group $\Lambda<G$ where $G= \prod_{i \in I} G_i$ where each $G_i$ is a rank one $p_i$-adic group whose
	Bruhat-Tits building is a tree. After some initial reductions, we will argue that $\Lambda$ is irreducible. First, by
{	Theorem~\ref{thm:flssimproved} (or \cite[Theorem 15]{FLSS})}, there is a subset $J \subset I$ such that the projection of $\Lambda$ to each $G_j$ is faithful and indiscrete where $\Lambda < G_J = \prod_{j \in J} G_j$ is discrete and $|J|\geq 2$.  We fix a subset $J$ that is minimal with respect to these properties.
	
	For $j\in J$, the $\bbQ_{p_j}$-Zariski closure of the projection of $\Lambda$ to $G_j$ is simple. Otherwise it would, up to finite index, fix a point at infinity, but {Theorem~\ref{thm:flssimproved}} shows that the projection of finite index subgroups of $\Lambda$ to each $G_j$ do not fix a {point} at infinity.
	
	We now replace every $G_j$ by the $\bbQ_{p_j}$-Zariski closure of the projection of $\Lambda$, and argue that $\Lambda\subseteq G_J$ is irreducible. Henceforth we will regard $G_j$ as a $\bbQ_{p_j}$-algebraic group, and omit the field when referring to its Zariski topology.
	
	Recall that for each prime $p$, the collection of $p$-adic factors of $G_J$ is indexed by $J_p$, and we write $G_p:=\prod_{j\in J_p} G_j$ for the product of the $p$-adic factors of $G_J$. By minimality of $J$, for every
	{$J'\subsetneq J$}, the projection of $\Lambda$ to $G_{J'}$ is indiscrete.
	
	Therefore it remains to show that for every prime $p$, the projection of $\Lambda$ to $G_p$ is Zariski dense. Let $H_p$ be the Zariski closure of the projection of $\Lambda$, and let $J_p'\subseteq J_p$ be those values of $j$ such that $G_j\subseteq H_p$. We will argue by contradiction that $J_p'=J_p$, and this will complete the proof.
	
	Suppose therefore that $J_p'$ is a proper subset of $J_p$.
	{Let $\overline{H_p} = H_p/G_{J_p'}$.}
	Consider now the image of { $\overline{H_p}$} in the product of the remaining factors $G_{J_p}/G_{J_p'}$. Since for every $j\in J_p$, the projection of $\Lambda$ to $G_j$ is Zariski dense, the image of { $\overline{H_p}$} surjects onto (but does not contain) $G_j$ for $j\in J_p\bs J_p'$. In particular, $G_{J_p}/G_{J_p'}$
	consists of at least two factors.
	
	Fix one such factor $G_{j_0}$. Since the kernel of the projection of { $\overline{H_p}$} to $G_{j_0}$ would be normal in the remaining factors, it is a product of some subset of them. But since { $\overline{H_p}$} does not contain any factors of $G_{J_p}/G_{J_p'}$, we conclude that { $\overline{H_p}$} projects isomorphically onto $G_{j_0}$.
	
{	Therefore the projection
$$\overline{H_p}\times G_{J_p'}\times G_J/G_{J_p} \to G_{j_0}\times G_{J_p'}\times G_J/G_{J_p}$$
 that replaces
  $\overline{H_p}$} by $G_{j_0}$ is a topological isomorphism. Since (the projection of) $\Lambda$ is discrete in the former, it is also discrete in the latter. However, since $J_p\bs J_p'$ consists of at least two factors, this contradicts minimality of $J$.

\end{proof}

{
With Corollary~\ref{noirreduciblesurfacesinpadics} in place, we now
complete the proof of Corollary~\ref{cor:algebraicsurfacegroups}.}
	{Note first that since $\Lambda$ is a surface group, it is torsion-free. Hence so is $\Gamma$. In particular, $\Gamma$ cannot be a non-trivial finite group.}

\begin{proof}[Proof of Corollary~\ref{cor:algebraicsurfacegroups}:]
	As $\Lambda$ is finitely generated, there are at most finitely many finite places $S$ of $k$ such that there is an element of $\Lambda$ of
	norm greater than one in the valuation.  Viewing $\Lambda$ as a subgroup of $G :=\prod_{s \in S} \PSL(2,k_s)$, we see that $\Gamma$ is the intersection
	of $\Lambda$ with a compact open subgroup in $G$. It follows immediately that $\Lambda$
	commensurates $\Gamma$ and that if $\Gamma$ is infinite, it is Zariski dense (by Lemma \ref{lemma:zdense-comm} and minimality of $k$). So it remains to show $\Gamma$ is infinite.
	Corollary \ref{noirreduciblesurfacesinpadics} implies that $\Lambda$ is not discrete in $G$ and so $\Gamma$ is infinite.
\end{proof}

{ 
\begin{rem}
Note that in the above proof, we allow for the possibility that the image of $\Lambda$ is bounded. In this case, $\Gamma$ is of finite index in $\Lambda$.
Indeed, if  $\Lambda$ itself is arithmetic, this is exactly what the above proof furnishes.
In this case, $\Gamma$ is also arithmetic.
	\end{rem}

\begin{rem}
For each of the places $s$, there is a compact open subgroup $Q_s \subset
\PSL(2,k_s)$. Also, if $\pr_s$ denotes the projection of $G$ onto $\PSL(2,k_s)$, let $\Gamma_s = \Lambda\cap \pr_s^{-1}(Q_s)$. Then $\Gamma = \cap_s \Gamma_s$. The subgroup  $\Gamma_s $ detects density at $s$. What the
proof of Corollary~\ref{cor:algebraicsurfacegroups} in fact shows is that
$\Gamma$ detects density simultaneously at all $s$.
	\end{rem}}

\section{Rank gradient and small generating sets} \label{sec:rk}

For a finitely generated group $\Gamma$, the \emph{rank} $\rk(\Gamma)$ is the minimal number of generators in a generating set for $\Gamma$. It is interesting to consider ranks of finite index subgroups of $\Gamma$. Writing $\Gamma$ as a quotient of a free group $F_r$ of rank $r:=\rk(\Gamma)$, the pre-image of a finite index subgroup $\Theta\subseteq \Gamma$ in $F_r$ is a free subgroup of index $[\Gamma:\Theta]$, and therefore has rank $[\Gamma:\Theta](\rk(\Gamma)-1)+1$. In particular,
	\begin{equation}
	\rk(\Theta)-1\leq (\rk(\Gamma)-1)[\Gamma:\Theta],
	\label{eq:rk-trivialupper}
	\end{equation}
so $\rk(\Theta)-1$ grows at most linearly in $[\Gamma:\Theta]$. Given a chain of finite index subgroups
	$$\Gamma = \Gamma_0\supseteq \Gamma_1 \supseteq \Gamma_2 \supseteq \dots,$$
the \emph{rank gradient} of $\Gamma$ with respect to the chain $(\Gamma_n)_n$ is
\[\rg(\Gamma, (\Gamma_n)_n):=\lim_{n\to\infty} \frac{\rk(\Gamma_n)-1}{[\Gamma:\Gamma_n]}.\]

In the context of lattices in semisimple Lie groups it seems possible that all statements about
rank gradient are actually consequences of a stronger property.

\begin{dfn}
We say $\Gamma$ has \emph{arbitrarily small $k$ generated finite index subgroups} if given any finite index subgroup
$\Gamma'<\Gamma$ there is a further finite index subgroup $\Gamma''$ that has a $k$ element
generating set.
\end{dfn}

In this section we explain some history of the notions mentioned above, some results on them and their connections
to both the Greenberg-Shalom hypothesis and one another.

{The main theorem of this section is Theorem~\ref{thm:rk}, which asserts that
	assuming  the Greenberg-Shalom hypothesis, the following holds. For $G=\prod_i G_i$  a product of $r\geq 2$ factors as in Standing assumptions \ref{assumptions}, and  $G_i$ corresponding to different places $p_i$ of $\bbQ$, ($p_i\leq\infty$), any irreducible lattice $\Gamma\subseteq G$ contains arbitrarily small $r$-generated finite index subgroups.}

\subsection{{Historical remarks and motivation}}
Lackenby introduced rank gradient as a generalization of Heegaard gradient, which is an invariant for 3-manifolds \cite{lackenby-rg}. Abert-Nikolov proved that if the chain $(\Gamma_n)_n$ is \emph{Farber} (e.g. if $\Gamma_n$ are contained in a chain of normal subgroups of $\Gamma$ that has trivial intersection), then the rank gradient computes the \emph{cost} of the action of $\Gamma$ on the boundary of the coset tree associated to the chain, i.e. the profinite space $\varprojlim \Gamma/\Gamma_n$ \cite{abert-nikolov-cost}. Cost is more generally defined for probability measure-preserving Borel actions of $\Gamma$, and the \emph{fixed price problem} asks whether the cost of an action only depends on $\Gamma$ and not on the action itself. A positive answer for many groups, including higher rank nonuniform irreducible lattices in real Lie groups has been given by Gaboriau \cite{gaboriau-cost}. Abert-Gelander-Nikolov have given a positive answer for right-angled lattices in such Lie groups, which includes the first uniform examples \cite{abert-gelander-nikolov}. {(Recall \cite{abert-gelander-nikolov} that a group $\Gamma$ is right-angled if it admits a finite generating set
$\{\gamma_1,\cdots, \gamma_n\}$ such that each $\gamma_i$ is non-torsion, and $\gamma_i, \gamma_{i+1}$ commute for $i = 1,\cdots, n-1$.)} Recently, Fraczyk-Mellick-Wilkens have positively answered the question for all lattices in higher rank real Lie groups \cite{fmw-cost}. In all of these results, the cost is 1, which implies vanishing of the rank gradient for any Farber sequence.  Some further results relevant here are contained in an even more recent preprint of Mellick \cite{Mellick}. It is possible that the strongest results on rank gradient that one can obtain conditionally on the Greenberg-Shalom hypothesis will soon be known by other methods.



However in the context of lattices in semisimple Lie groups, oftentimes one can find many finite index subgroups with a uniformly bounded number of generators. This is not the case for lattices in $\SL(2,\bbR)$, since those surject onto free groups, and so do some lattices in other rank 1 groups. However,by combining work of Raghunathan, Tits, and Venkataramana, one can prove that for a nonuniform higher rank arithmetic lattice $\Gamma$ in a $\bbQ$-simple real Lie group $G$, there exists $k\geq 2$ such that $\Gamma$ has arbitrarily small $k$-generated finite index subgroups  (see \cite[Remark 3]{sv-rk}). In particular, the cost of Farber chains is always 1. Sarma-Venkataramana have shown that one can take $k=3$ \cite{sv-rk}. Sarma has studied the same problem for $S$-arithmetic lattices, and in particular has proven that for an $S$-arithmetic lattice in $\SL(2,K)$, where $K$ is a number field and $|S|\geq 2$, one can take $k=3$ \cite{sarma-rk}. For $\SL(n,\bbZ), n\geq 3,$ Lubotzky has asked whether one can take $k=2$, i.e. whether $\SL(n,\bbZ), n\geq 3,$ has ``arbitrarily small 2-generator finite index subgroups'' \cite{lubotzky-dimfn}, and this was proven by Meiri \cite{meiri-rk}. The analogue of Lubotzky's question is open for all other commensurability classes of higher rank  lattices, including $S$-arithmetic ones.

\subsection{{The Greenberg-Shalom hypothesis and rank}}
We start with the following observation.
\begin{rem} Suppose that the Greenberg-Shalom Hypothesis \ref{q:shalom} holds.
	Then	for $N\geq 2$ and $\Gamma=\SL(2,\bbZ[1/N])$,
	the Lyndon-Ullman group $\Delta_{1/N}$ is a lattice in $\SL(2,\bbR)\times \prod_{p\mid N} \SL(2,\bbQ_p)$ (see Theorem \ref{thm:lukk}).
	It is therefore a finite index 2-generated subgroup of $\SL(2,\bbZ[1/N])$.  This shows that the Greenberg-Shalom Hypothesis in at least some cases implies $k=2$ even when the number of places is greater than 2.  \end{rem}

Heuristically, under the Greenberg-Shalom hypothesis, for irreducible lattices $\Gamma\subseteq G$ (where $G$ consists of at least 2 factors) one should indeed have $k=2$, and in particular, the cost of Farber chains is always 1: namely, it is plausible that  a generic pair of elements generates an irreducible group, which would have finite index by Proposition \ref{prop-irred2latt}.
The following result shows that one can indeed make such an argument (with $k$ given by the number of factors):

\begin{thm} Assume the Greenberg-Shalom hypothesis. Let $G=\prod_i G_i$ be a product of $r\geq 2$ factors as in Standing assumptions \ref{assumptions}, and assume $G_i$ correspond to different places $p_i$ of $\bbQ$, {with $p_i\leq\infty$}. Then any irreducible lattice $\Gamma\subseteq G$ contains arbitrarily small $r$-generated finite index subgroups.
	\label{thm:rk}
\end{thm}

\begin{cor}
	\label{cor:rk} Assume the Greenberg-Shalom hypothesis.
	 $\Gamma=\SL(n,\bbZ[1/p])$ contains arbitrarily small 2-generated finite index subgroups.
\end{cor}

To prove Theorem \ref{thm:rk}, we start with some preliminaries. First, we recall some facts about Jordan projections and loxodromic elements (see e.g. \cite[Chapter 6]{benoistquint-book}). For a connected simple algebraic group $H$ over a local field of characteristic 0, every nontrivial element admits a Jordan decomposition $g= g_e g_h g_u$ as a product of a commuting triple consisting of an elliptic element $g_e$, a hyperbolic element $g_h$, and a unipotent element $g_u$. An element $g\in H$ is called \emph{loxodromic} if $g_h$ is regular, i.e. conjugate into $\exp \fra^+$, where $\fra^+$ is the (open) positive Weyl chamber. Any loxodromic element $g$ is semisimple, i.e. $g_u=e$, and is therefore conjugate to an element of the form $m \exp(X)$ where $X\in\fra^+$ and $m\in Z_K(A)$ centralizes the maximal torus $A=\exp\fra$. If $H$ is real, then for any Zariski dense subgroup $\Theta\subseteq H$, the set $\Theta_{\text{lox}}$ of loxodromic elements is also Zariski dense (and, in particular, nonempty) in $H$ (see e.g. \cite[Theorem 6.36]{benoistquint-book}). If $H$ is $p$-adic, this is true under the additional assumption that every simple root of $H$ is unbounded on the set of Cartan projections of $\Theta$ (see \cite[Lemma 9.2]{benoistquint-book}). These existence results will be useful to us because, by the following lemma, a loxodromic element and a generic element generate a Zariski dense subgroup.

\begin{lemma} Let $H$ be a connected simple algebraic group over a local field of characteristic 0, and let $h\in H$ be nontrivial. Assume that either
	\begin{enumerate}[(i)]
		\item $h$ is loxodromic, or
		\item $H$ is nonarchimedean and $h$ belongs to a compact, open, torsion-free subgroup that is contained in the image of the exponential map.
	\end{enumerate}
 Then there exists a nonempty Zariski open (and in particular, open and dense) set $S\subseteq H$ such that for $g\in S$, the group $\langle g,h\rangle$ is Zariski dense in $H$.
\label{lemma:Z-generation}
\end{lemma}

We postpone the proof of Lemma~\ref{lemma:Z-generation} till after that of
Theorem~\ref{thm:rk}.
%
Henceforth we will refer to the $\bbQ_{p_i}$-Zariski topology on $G_i$ simply as the Zariski topology, and likewise for the product Zariski topology on $G=\prod_i G_i$.

To find elements in the Zariski open sets given by the above lemma, we need to establish Zariski density of suitable subsets of $\Gamma$. This is accomplished by the following lemma:

\begin{lemma} Assume the notation of Theorem \ref{thm:rk}. Fix $j\in I$ and let $U\subseteq G/G_j$ be an open neighborhood of identity. Denote by $\Gamma_j\subseteq\Gamma$ the subset consisting of elements $\gamma\in\Gamma$ whose projection mod $G_j$ lies in $U$.

Then
	\begin{enumerate}[(i)]
		\item $\pr_j(\Gamma_j)$ is Zariski dense in $G_j$, and
		\item If $G_j$ is $p$-adic for some $p$, then every simple root of $G_j$ is unbounded on the Cartan projections of $\pr_j(\Gamma_j)$.
	\end{enumerate}
\label{lemma:independence}
\end{lemma}

We also postpone the proof of Lemma~\ref{lemma:independence} till after that of
Theorem~\ref{thm:rk}.
With these preliminary facts stated, we can start the proof of Theorem \ref{thm:rk} proper:

\begin{proof}[Proof of Theorem \ref{thm:rk}] Index the factors by a set of primes and possibly $\infty$. If there is no archimedean factor, relabel one of the finite primes (formally) to be $\infty$. Then for $p\neq \infty$, let $K_p\subseteq G_p$ be a compact open subgroup contained in the image of the exponential map on $G_p$ and set $K:=\prod_{p\neq \infty} K_p$. For any $p$, let $q_p:G\to G/G_p\cong\prod_{\ell\neq p} G_\ell$ be the quotient mod $G_p$, and define the subgroup $\Gamma_K:=\Gamma\cap q_\infty^{-1}(K)$ consisting of those elements of $\Gamma$ whose projection mod $G_\infty$ lies in $K$. By Lemma \ref{lemma:independence}, $\pr_\infty(\Gamma_K)$ is Zariski dense in $G_\infty$, and hence contains a loxodromic element. Let $a\in\Gamma_K$ be such that $\pr_\infty(a)$ is loxodromic.

We will now show that for every $p\neq \infty$, there exists $b_p\in\Gamma$ such that
	\begin{enumerate}[(i)]
		\item $\langle a,b_p\rangle$ is indiscrete mod $G_p$, and
		\item $\langle a,b_p\rangle$ has Zariski dense projections to $G_\infty$ and to $G_p$.
	\end{enumerate}
Assuming existence of such $b_p$, it is easy to see $\langle a,b_p\mid p\neq \infty\rangle$ is irreducible in $G$. Indeed, the projections mod $G_p$ are indiscrete for $p\neq \infty$ by Property (i). The projection mod $G_\infty$ is indiscrete because $q_\infty(a)$ belongs to the compact torsion-free subgroup $K$. Finally, projections to any factor are Zariski dense by Property (ii).

So it remains to establish existence of elements $b_p\in\Gamma$ with the above Properties (i) and (ii). By Lemma \ref{lemma:Z-generation} above, for every factor $G_\ell$ (including $G_\infty$) there exists an open dense set $S_\ell \subseteq G_\ell$ such that for any $g_\ell \in S_\ell$, the group $\langle g_\ell, \pr_\ell(a)\rangle$ is Zariski dense in $G_\ell$. For every $p\neq \infty$, the set $\prod_{\ell \neq p} S_\ell\subseteq \prod_{\ell\neq p} G_\ell$ is open and dense, so we can choose $g_p\in \prod_{\ell\neq p} S_\ell$ sufficiently close to identity so that iterated commutators of $g_p$ and $q_p(a)$ converge to the identity in $G/G_p$. Since the subsets $S_\ell\subseteq G_\ell, \ell\neq p,$ are open, there is an open neighborhood $V_p$ of $g_p \in \prod_{\ell\neq p} G_\ell$ such that these properties hold for any element of $V_p$. To summarize, we have that for any $g \in V_p$, the group $\langle g,q_p(a)\rangle$ is indiscrete in $G/G_p$, and has Zariski dense projection to $G_\ell$ for all $\ell\neq p$.

Since $\Gamma$ has dense projection to $\prod_{\ell\neq p} G_\ell$, we can choose  $\gamma_p\in\Gamma$ such that $q_p(\gamma_p)\in V_p$. Further choose $W_p\ni e$ an open neighborhood of the identity in $G/G_p$ such that $q_p(\gamma_p) W_p\subseteq V_p$. Let $\Gamma_p:=\Gamma \cap q_p^{-1}(W_p)$ consist of those elements of $\Gamma$ whose projection mod $G_p$ lies in $W_p$.

We now consider elements $b_p$ of the form $\gamma_p \eta_p$, where $\eta_p\in \Gamma_p$. For such elements, the projection to $G_\infty$ lies in $S_\infty$, so $\langle a,b_p\rangle$ has Zariski dense projection to $G_\infty$. Further, iterated commutators of $q_p(b_p)$ and $q_p(a)$ converge to the identity in $G/G_p$, and they cannot terminate since then $\langle a,b_p\rangle$ would be nilpotent and in particular not Zariski dense in $G_\infty$. This establishes both claimed Properties (i) and (ii) of $b_p$ except the Zariski density of the projection to $G_p$, so we will now show we can choose $b_p$ to guarantee this as well.

This last property exactly means that we require $\pr_p(b_p)\in S_p$. On the other hand, the possible choices of $\pr_p(b_p)$ are among $\pr_p(\gamma_p)\pr_p(W_p)$. This latter set is Zariski dense by Lemma \ref{lemma:independence} (applied with $U=W_p$), and since $S_p$ is nonempty and Zariski open, we have $S_p\cap \pr_p(\gamma_p) \pr_p(W_p)\neq \varnothing$, so that there exists a satisfactory choice of $b_p$. \end{proof}

{Corollary~\ref{cor:rk} immediately follows from the proof of Theorem~\ref{thm:rk} above.}\\

\subsection{{Proofs of Lemma~\ref{lemma:Z-generation}
and  Lemma~\ref{lemma:independence}}}
{We now furnish the postponed proofs of Lemma~\ref{lemma:Z-generation}
and  Lemma~\ref{lemma:independence}.}
\begin{proof}[Proof of Lemma~\ref{lemma:Z-generation}:] The idea is to consider the set of elements $g$ such that $\Ad(g)$ and $\Ad(h)$ have no joint invariant subspace in the Lie algebra $\frh$ of $H$. Then for such $g$, as long as the Lie algebra of $\overline{\langle g,h\rangle}^Z$ is nontrivial, it must be all of $\frh$, and this will imply $\langle g,h\rangle$ is Zariski dense. However, we will actually work with a smaller set of $g$ that we can prove is Zariski open.
	
	To define this set, for $g\in H$ and $0<d<\dim\frh$, let $I_g(d)$ denote the collection of $d$-dimensional $\Ad(g)$-invariant subspaces of $\frh$, and let $I_g:=\cup_d I_g(d)$ denote the collection of all proper $\Ad(g)$-invariant subspaces of $\frh$. Note that $I_g(d)$ is a Zariski closed subset of the Grassmannian $\Gr_d(\frh)$. Let $S_d$ be the set of all elements $g\in H$ such that $\Ad(g)I_h(d)\cap I_h(d)=\varnothing$, so the complement of $S_d$ consists of those $g\in H$ such that there exists $x\in I_h(d)$ with $gx\in I_h(d)$. Note also that, as promised, we indeed have $I_g(d)\cap I_h(d)=\varnothing$ for all $g\in S_d$.
	
	Let
	$$\pi_H: H\times \Gr_d(\frh)\times \Gr_d(\frh)\to H$$
	be the projection onto the first factor. In terms of this map, we have
	$$S_d^c = \pi_H(\{(g,x,gx)\in H\times \Gr_d(\frh)\times \Gr_d(\frh)\mid x, gx\in I_h(d)\}).$$
	Let $\alpha:H\times \Gr_d(\frh)\to \Gr_d(\frh)$ denote the action map $\alpha(g,x):=gx$. Then we have
	$$S_d^c = \pi_H(\text{Graph}(\alpha)\cap (H\times I_d(h)\times I_d(h))).$$
	Using this description, we will show that $S_d^c$ is Zariski closed. A projection $X\times Y\to X$ of varieties is a closed map (with respect to the Zariski topology on $X\times Y$) if $Y$ is a projective variety, so $\pi_H$ is a closed map. Of course $H\times I_d(h)\times I_d(h)$ is Zariski closed, so it remains to show Graph$(\alpha)$ is Zariski closed. This is a special case of the general fact that the graph of a surjective algebraic map $F:X\to Y$ to a projective variety $Y$ is Zariski closed in $X\times Y$. We have shown $S_d$ is Zariski open, and therefore so is $S=\cap_d S_d$.
	
	Next, we show that $S\neq \varnothing$. For this, we use the following property of $h$ that holds under the assumption that either $h$ is loxodromic or sufficiently close to identity:  there exists $X\in\frh$ such that $I_h$ is contained in the collection $I_X$ of $\ad(X)$-invariant proper subspaces of $\frh$. We prove this by considering the cases that (i) $h$ is loxodromic, and (ii) $h$ is contained in a compact open torsion-free subgroup.
	
	In Case (i), $h$ is conjugate by some $c\in H$ to an element of the form $m \exp(X)$, where $m$ is elliptic and centralizes $A$, and $X\in\fra^+$ is regular. By inspecting the action of $\Ad(m)$ and $\Ad(\exp X)$ relative to the root space decomposition of $\frh$, we see that $I_{m\exp(X)}\subseteq I_{\exp X}$ and also that $I_{\exp X} = I_X$. Therefore $I_h \subseteq  I_{c^{-1} \exp(X) c}$. We can write $c^{-1} \exp(X) c = \exp(\Ad(c)^{-1} X)$, and therefore $I_h \subseteq I_{\Ad(c)^{-1}X}$. This completes Case (i).
	
	In Case (ii), we can write $h=\exp(X)$. Then of course any $\ad(X)$-invariant subspace is $\Ad(h)$-invariant. For the converse, note that if $H$ is $p$-adic, we have $h^{p^n}=\exp(p^n X)\overset{n\to\infty}{\longrightarrow} \Id$, and
	$$\ad(X) v = \lim_{n\to \infty} \frac{1}{p^n}\left(\Ad(\exp(p^n X))-\Id\right) v.$$
	Now let $V$ be an $\Ad(h)$-invariant subspace. Then $V$ is also $\Ad(h^{p^n})$-invariant for any $n$, and by the above formula, is also $\ad(X)$-invariant.
	
	Finally, by a theorem of Bois \cite{bois-15}, the Lie algebra $\frh$ is 1.5-generated, i.e. for every $0\neq A\in\frh$ there exists $B\in\frh$ such that $\{A,B\}$ generates $\frh$ as a Lie algebra. Further, the set of $B$ with this property is Zariski open (see \cite[Proposition 1.1.3]{bois-15}). Choose $Y$ such that $\{X,Y\}$ generates $\frh$, so that $I_X\cap I_Y=\varnothing$. Since the set of possible choices of $Y$ is Zariski open, we can choose such $Y$ that is loxodromic, i.e. $Y=E+A$ where $A$ is regular and $E$ is elliptic centralizing the maximal torus containing $A$. As we commented before, for such elements we have $I_{\exp Y} \subseteq I_A$. Since $\exp(Y)$ is loxodromic, $\Ad\exp Y$ contracts towards $I_Y$, so that after possibly replacing $Y$ by a large scalar multiple, $\Ad(\exp Y)I_X$ is contained in a neighborhood of $I_Y$ disjoint from $I_X$, so $\exp Y\in S$.
	
	
	
	It remains to show that for $g\in S$, the group $\langle g,h\rangle\subseteq H$ is Zariski dense. Let $g\in S$ and note that $\langle g,h\rangle$ is infinite (because $h$ has infinite order), and hence its Zariski closure has nontrivial Lie algebra, which is of course $\langle \Ad(g),\Ad(h)\rangle$-invariant. It follows that the Lie algebra of $\overline{\langle g,h\rangle}^Z$ is all of $\frh$, i.e. $\overline{\langle g,h\rangle}^Z$ is open. However, any proper Zariski closed subgroup is nowhere dense, so we must have $\overline{\langle g,h\rangle}^Z=H$. \end{proof}

	\begin{proof}[Proof of Lemma~\ref{lemma:independence}:] $\mbox{}$\\ Proof of (i): As $U$ decreases, so does the Zariski closure of $\pr_j(\Gamma_j)$ in $G_j$. By the Noetherian property of $G_j$ {(Lemma~\ref{lem-noeth})}, these Zariski closures eventually stabilize, so we can assume stabilization has occurred at $U$, and denote the corresponding Zariski closure by $H_j\subseteq G_j$. Then choosing $V\subseteq U$ symmetric such that $V^2\subseteq U$, it is easy to see that $H_j^2\subseteq H_j$ and $H_j^{-1}=H_j$, i.e. $H_j$ is a group. Further, if $\gamma\in\Gamma$ is fixed, we can choose $V_\gamma\subseteq U$ such that $\gamma V_\gamma \gamma^{-1}\subseteq U$. It follows that $H_j$ is normalized by $\pr_j(\gamma)$. Since $\gamma\in\Gamma$ was arbitrary, we conclude that $H_j$ is normalized by $\pr_j(\Gamma)$, which is Zariski dense in $G_j$. Since $H_j$ is Zariski closed, this implies that $H_j$ is normal in $G_j$, and since $H_j$ is infinite, we must have $G_j^+\subseteq H_j$. Finally, since $G_j^+$ is Zariski dense (see Theorem \ref{thm:G+}.(i)), we conclude that so is $H_j$.
		
		Proof of (ii): Let $\kappa_j(\Gamma_j)$ be the set of Cartan projections of $\pr_j(\Gamma_j)$, and let $\alpha$ be any simple root of $G_j$. We need to show that $\alpha$ is unbounded on $\kappa_j(\Gamma_j)$. First note that $\pr_j(\Gamma_j)$ is unbounded (since any sequence of elements whose projections converge to the identity in $G/G_j$ must diverge in $G_j$). Therefore $\kappa_j(\Gamma_j)$ is unbounded as well, and since the Weyl group orbit of $\alpha$ spans the dual $\fra^\ast$, there exists a Weyl group element $w\in W$ such that $w\alpha$ is unbounded on $\kappa_j(\Gamma_j)$.
		
		By a result of Tits \cite{tits-weylsubgroup}, there exists a finite extension $\widetilde{W}$ of $W$ contained in $G$ such that if $\widetilde{w}\in \widetilde{W}$ has image $w\in W$, then $\Ad(\widetilde{w})$ coincides with the Weyl group action of $w$ on $\fra$. After possibly conjugating by an element of $A$, we can assume $\widetilde{W}$ is contained in the maximal compact subgroup $K_j$ of $G_j$. Since $\Gamma$ has dense projection to $G_j$ and $K_j$ is open, we can choose $\gamma\in\Gamma$ such that $\pr_j(\gamma)\in \widetilde{w}^{-1} K_j$, say $\pr_j(\gamma) = \widetilde{w}^{-1} k$ for some $k\in K_j$. Now choose $\gamma_n\in\Gamma_j$ converging to the identity in $G_j$ and such that $(w\alpha)(\kappa_j(\gamma_n))\to\infty$. For $n\gg 1$, we have $q_j(\gamma\gamma_n\gamma^{-1})\in U$, so $\gamma \gamma_n \gamma^{-1} \in \Gamma_j$. Let
		$$\pr_j(\gamma_n) = l_n a_n l_n'$$
		be the Cartan decomposition of $\pr_j(\gamma_n)$. Then we simply compute
		$$\pr_j(\gamma \gamma_n \gamma^{-1}) = \widetilde{w}^{-1} k l_n a_n l_n' k^{-1} \widetilde{w}.$$
		Using that $\widetilde{W}\subseteq K_j$, we see that the Cartan projection of $\pr_j(\gamma \gamma_n \gamma^{-1})$ is given by
		$$\kappa_j(\gamma \gamma_n \gamma^{-1}) = \widetilde{w}^{-1} a_n \widetilde{w}.$$
		and
		$$\alpha(\widetilde{w}^{-1} a_n \widetilde{w}) = (w\alpha)(\gamma_n) \to \infty.$$
	\end{proof}
	
%
%


\subsection{{Application and question}}
\begin{question} Let $G$ and $\Gamma$ be as in Theorem~\ref{thm:rk}, and assume $r>2$. Does $\Gamma$ even contain a single $k$-generated finite index subgroup for some $k<r$? Does $\Gamma$ admit arbitrarily small $k$-generated subgroups for some $k<r$? E.g. for $k=3$ as in the results of Sarma-Venkataramana and Sarma mentioned above? Or even $k=2$ as in Lubotzky's question and Meiri's theorem? \end{question}

An elementary argument shows that if a chain of finite index subgroups $(\Gamma_n)_n$ admits a refinement with uniformly bounded rank (in fact $o([\Gamma:\Gamma_n])$ suffices), then $(\Gamma_n)_n$ has vanishing rank gradient. We include it for completeness:
\begin{prop} Let $\Gamma$ be a finitely generated group. {Let $(\Gamma_n)_n$
	be a decreasing chain of finite index normal subgroups.}
Let $(\Gamma_n')$ be a decreasing chain of finite index subgroups such that $\Gamma_n'\subseteq\Gamma_n$ for all $n$, and $\rk(\Gamma_n')\leq r$ is uniformly bounded.
	Then the rank gradient of $(\Gamma_n)_n$ vanishes.\end{prop}
\begin{proof} Let $\varepsilon>0$ and let $m \geq n \geq 1$ be sufficiently large (we will specify how large in a moment). First we bound the rank of $\Gamma_n'\cap \Gamma_m$ using the trivial bound as a finite index subgroup of $\Gamma_n'$ (see Equation \eqref{eq:rk-trivialupper}), and we find
	\begin{align*}
		\rk(\Gamma_n'\cap \Gamma_m)-1	&\leq 	(r-1) [\Gamma_n':\Gamma_n'\cap \Gamma_m] 	\\
		&\leq (r-1) [\Gamma_n:\Gamma_m],						\\
		&\leq  \frac{ (r-1)}{[\Gamma:\Gamma_n]} [\Gamma:\Gamma_m],
	\end{align*}
	where for the second estimate we used that $m\geq n$ so that $\Gamma_m\subseteq\Gamma_n$ and $\Gamma_n'/(\Gamma_n'\cap\Gamma_m)\hookrightarrow \Gamma_n/\Gamma_m$. In particular, by choosing $n\geq 1$ such that $(r-1)[\Gamma:\Gamma_n]^{-1}<\varepsilon$, we see $\Gamma_n'\cap\Gamma_m$ is generated by at most $\varepsilon[\Gamma:\Gamma_m]$ many elements. Now we need to show that by not adding too many generators, we can actually generate all of $\Gamma_m$. For this we simply use the trivial bound:
	$$\rk(\Gamma_m)\leq [\Gamma_m:\Gamma_n'\cap\Gamma_m] + \rk(\Gamma_n'\cap\Gamma_m).$$
	Since $\Gamma_m/(\Gamma_n'\cap\Gamma_m)\hookrightarrow \Gamma/\Gamma_n'$, we can estimate the first term by $[\Gamma:\Gamma_n']$. So far we have only used $m\geq n$, but now take $m$ sufficiently large so that $[\Gamma:\Gamma_m]>\varepsilon^{-1} [\Gamma:\Gamma_n']$. Then we have $\rk(\Gamma_m) \leq 2\varepsilon [\Gamma:\Gamma_m]$, as desired.\end{proof}

Therefore we have the following application of Theorem \ref{thm:rk}:
\begin{cor} Assume the Greenberg-Shalom hypothesis, and let $G$ be as in Standing Assumptions \ref{assumptions} with simple factors corresponding to different places of $\bbQ$, and $\Gamma\subseteq G$ an irreducible lattice. Then the rank gradient of $\Gamma$ with respect to any infinite chain of finite index normal subgroups vanishes.
	\label{cor:rg}\end{cor}
As mentioned earlier, Fraczyk-Mellick-Wilkens have proven this result unconditionally for higher rank lattices in real Lie groups and Mellick has extended this to the case of at least some products with rank one factors. It seems plausible (but we do not know) that their methods extend to the setting of lattices as in the corollary as well.  Those techniques do not seem capable of addressing the more refined question of arbitrarily small $k$-generated finite index subgroups.


\appendix
\section{Background on algebraic groups, irreducibility and Venkataramana's Lemma}

We collect in this appendix a variety of results and proofs all of which depend heavily
on the theory of algebraic groups.  We hope this makes the main body of the paper more
readable for audiences not familiar with these ideas.

We begin with a pair of important facts that drive our entire exploration.  One
is classical, the other elementary.

\begin{prop}\label{compactopencomm}
Every totally disconnected locally compact group $G$ contains a compact open subgroup $K$.
Furthermore $G$ commensurates $K$.
\end{prop}

\begin{proof}
The first statement is van Dantzig's theorem see e.g. \cite[Theorem 16]{Pon}.  The second statement can be proven by covering any
conjugate of $K$ by finitely many $K$ cosets, using that $K$ is compact and open.
\end{proof}

\subsection{Preliminaries on simple algebraic groups} \label{subsec:prelim}
We will now review classical results in the study of simple algebraic groups  used in the main body of the paper, and end by proving that under Standing Assumptions \ref{assumptions}, the projections of $\Gamma$ are Zariski dense. We start by fixing some notation:

\begin{assume} For the rest of this section, $k$ denotes a local field of characteristic zero, and $\bbH$ denotes a connected semisimple adjoint algebraic group defined over $k$ such that $\bbH(k)$ does not have compact factors.\end{assume}

We write $H:=\bbH(k)$. Let us start with the following useful definition:
\begin{dfn} $H^+$ denotes the group generated by unipotent elements of $H$.\end{dfn}
It is not necessary to assume $k$ is a local field of characteristic zero, but the definition of $H^+$ is more difficult in general. To simplify the discussion and since we only need this notion for local fields of characteristic zero, we will only give this definition (but see e.g. \cite[Section I.1.5]{margulis-book}). We have the following result due to Borel-Tits:

\begin{thm}[{Borel-Tits \cite[6.7, 6.9, and 6.14]{borel-tits-morphismes-abstraits}}] \label{thm:G+} \mbox{}
	\begin{enumerate}[(i)]
		\item $H^+\subseteq H$ is Zariski dense and (analytically) open. In particular $H^+$ has finite index in $H$.
		\item $H^+$ does not contain any proper subgroup of finite index. Hence every finite index subgroup of $H$ contains $H^+$.
	\end{enumerate}\end{thm}
If $k=\bbR$, then $H^+$ is the connected component of $H$ that contains the identity (see again \cite[6.14]{borel-tits-morphismes-abstraits}). In particular, any open subgroup of $H$ contains $H^+$. A nonarchimedean analogue is the following unpublished result of Tits, with a published proof due to Prasad:
\begin{thm}[{Tits-Prasad \cite{prasad-tits}}] \label{thm:tits-prasad}
Assume in addition that $\bbH$ is almost simple. Then any noncompact, open subgroup of $H$ contains $H^+$.\end{thm}

\noindent { In view of the above results, a subgroup $\Delta<H$ is almost dense if and only if its closure contains $H^+$. Since the analytic closure is contained in the algebraic closure and $H^+$ is Zariski dense, it follows that any almost dense group is Zariski dense.}

{The following general and basic fact will be used a number of times in the sequel. We therefore separate it out for easy reference.}

{ 
	\begin{lemma}\label{lem-noeth}
	Let $G$ be an algebraic group. Let $\Gamma_1 \supset \Gamma_2
	\supset \cdots $ be a descending sequence of (possibly discrete) subgroups of $G$. Let ${\overline{\Gamma_i}}^Z$ denote the Zariski closure of $\Gamma_i$ in $G$. Then, there exists $N \in \natls$ such that
	${\overline{\Gamma_i}}^Z={\overline{\Gamma_j}}^Z$ for all $i, j \geq N$.
	\end{lemma}
	
	\begin{proof}
	Since $\Gamma_1 \supset \Gamma_2
	\supset \cdots $, it follows that ${\overline{\Gamma_i}}^Z \supseteq {\overline{\Gamma_{i+1}}}^Z$ for all $i \geq 1$. Since each ${\overline{\Gamma_i}}^Z $ is Zariski closed, the Lemma follows from the
	Noetherian property for Zariski closed subsets of $G$.
	\end{proof}
}

Now, let us return to the situation of groups with almost dense commensurators and prove the following straightforward lemma:
\begin{lemma} \label{lemma:zdense-comm} Suppose $\Theta\subseteq H$ is an infinite subgroup with Zariski dense commensurator $\Delta$. Then, possibly after passing to a finite index subgroup,  there is a subset of factors $H_J\subseteq H$ containing $\Theta$ and $\Theta\subseteq H_J$ is Zariski dense.\end{lemma}

\begin{rem} In the above lemma, we do not assume that $\Theta$ is a discrete subgroup of $H$.\end{rem}

\begin{proof} Consider the Zariski closures of finite index subgroups of $\Theta$, partially ordered by inclusion. We say $\Theta'\subseteq\Theta$ is Zariski dense in $\Theta$ if $\overline{\Theta'}^Z=\overline{\Theta}^Z$.  By (Lemma~\ref{lem-noeth}), there is a finite index subgroup $\Theta'$ of $\Theta$ such that any finite index subgroup $\Theta''\subseteq \Theta'$ is Zariski dense in $\Theta'$. Let $L$ be the Zariski closure of $\Theta'$ in $H$.

We claim that $\Delta$ normalizes $L$: Indeed, for any $\delta\in\Delta$, there is some finite index subgroup $\Theta'_\delta\subseteq\Theta'$ with $\delta\Theta_\delta'\delta^{-1} \subseteq \Theta'$. Since both $\Theta'_\delta$ and $\delta\Theta'_\delta\delta^{-1}$ are Zariski dense in $\Theta'$, it follows that $\delta$ normalizes $L$. Since $\Delta$ is Zariski dense and $L$ is Zariski closed, it follows that $H$ normalizes $L$. Since $H$ is adjoint, we have $L=H_J$ for some collection of factors $J$.

Finally, we conclude that $\Theta$ is contained in $H_J$: Indeed, the image of $\Theta$ in the remaining factors $H_{J^c}$ is finite and normalized by the projection $\pr_{J^c}(\Delta)$. A finite index subgroup $\Delta'$ of $\Delta$ has projection to $H_{J^c}$ that centralizes the image of $\Theta$ and is still Zariski dense, so that the image of $\Theta$ is central in $H_{J^c}$. Finally, since $H$ is adjoint, $H_{J^c}$ is centerless.\end{proof}

\begin{rem} The proof does not use the analytic topology on $k$, so the statement is still true if $k$ is a number field. \label{rem:nrfield} \end{rem}

%
{
\begin{lemma}\label{rmk:gamma-zdense} If $G$ is as in Standing Assumptions \ref{assumptions} and $\Gamma$ is a discrete subgroup with unbounded projection to each factor and almost dense commensurator, then for every $i\in I$, the projection $\pr_i(\Gamma)\subseteq G_i$ is Zariski dense. \end{lemma}

\begin{proof}

By assumption $\pr_i(\Gamma)$ has almost dense (and hence Zariski dense) commensurator. Since $\pr_i(\Gamma)$ is unbounded, it is infinite, so Lemma \ref{lemma:zdense-comm} implies it is Zariski dense.
\end{proof}

We end this subsection by giving the general proof of our key lemma relating irreducibility and commensurators.

\begin{proof}[Proof of Lemma \ref{lem:intersectionprojection}]
Since $G_{J^c} \times K$ is commensurated by $G$, we see that $\Gamma_K$ is commensurated by $\Gamma$.
Since $\Gamma$ is discrete in $G$, we see that $\Gamma_K$ is discrete in $G_{J^c} \times K$ and
since $K$ is compact $\pi_{J^c}(\Gamma_K)$ is discrete in $G_{J^c}$.  Let $I \subset J^c$ be
any subset of indices.  Assuming $\pi_{I}(\Gamma_K)$ is discrete, we will show that this implies
$\pi_{I\times J}(\Gamma)$ is discrete, contradicting irreducibility.    Take a sequence
$\gamma_i \in \Gamma$ such that $\pi_{I \times J}(\gamma_i) \rightarrow 1$.  Then
it is also true that $\pr_J(\gamma_i) \rightarrow 1$ or 
that $\gamma_i$ is eventually in $\Gamma_K$.  Since we assume $\Gamma_K$ is discrete in $G_{I \times J}$
this implies that $\gamma_i$ is eventually the identity, the desired contradiction to irreducibility.

Now we consider the Zariski closure $L$ of $\pr_{J^c}(\Gamma_K)$.  Since $\Gamma$ has dense projection
in $G_J$ and commensurates $\Gamma_K$, we see $L$ is normalized by $G_J$.  By Lemma \ref{lemma:zdense-comm},
 $L$ is a product of factors of $G_J$.  Thus there is an isomorphic projection of $L$ onto some subcollection of factors $G_I$ of $G_J$.
Since $\Gamma_K$ is discrete and contained in its own Zariski closure, this implies the projection
of $\Gamma_K$ to $G_I$ is discrete.  As above this is a contradiction unless $I=J$, in which case
$\Gamma_K$ is Zariski dense.
\end{proof}



\subsection{Further theory of $p$-adic Lie groups.}\label{subsec:padic}

We will now discuss some standard material on $p$-adic Lie groups that we need. A good reference containing this material is \cite{serre-liebook}. For $p$ a prime, the field $\bbQ_p$ of $p$-adic numbers is the completion of $\bbQ$ with respect to the norm $\|x p^k\|_p := p^{-k}$ where $k\in\bbZ$ and $x\in\bbQ$ has numerator and denominator coprime to $p$. $\bbQ_p$ is a locally compact field of characteristic zero.

A $p$-adic analytic manifold of dimension $n$ is a topological space with an atlas of charts valued in $\bbQ_p^n$, such that transition functions are analytic. It is a $p$-adic Lie group if group multiplication and inversion are analytic maps. In particular $\bbQ_p^n$ (as well as its open subgroup $\bbZ_p^n$) and $\GL(n,\bbQ_p)$ are $p$-adic analytic groups.  Several pertinent results behave as for real Lie groups:

\begin{thm}[Cartan's subgroup theorem {\cite[Corollary, p. 155]{serre-liebook}}]
Any closed subgroup of a $p$-adic Lie group is itself a $p$-adic Lie group.
\end{thm}

In particular, the $\bbQ_p$-points of an algebraic group, which is a closed subgroup of $\GL(N,\bbQ_p)$ for some $N$, forms a $p$-adic Lie group.

Further, $p$-adic groups have Lie algebras.  See \cite[Section V.1]{serre-liebook} for the construction of Lie algebras for $p$-adic groups.

A key difference between $p$-adic and real Lie groups is the structure of open neighborhoods of identity: For connected real Lie groups, any open neighborhood of identity generates the entire group. However, in the $p$-adic setting, already $\bbZ_p\subseteq \bbQ_p$ is an open subgroup. Furthermore, the decreasing sequence of subgroups $p^k \bbZ_p, \, k\geq 0$, form a basis of the topology at the identity. Similarly, $\GL(n,\bbZ_p)\subseteq \GL(n,\bbQ_p)$ is a compact open subgroup, and the congruence subgroups $\GL(n,\bbZ_p)[p^k]:=\ker(\GL(n,\bbZ_p)\to \GL(n,\bbZ/p^k\bbZ))$ form a basis of compact open subgroups at the identity.

Therefore any $p$-adic linear group $G$ admits a basis at the identity consisting of compact open subgroups, namely $G\cap \GL(n,\bbZ_p)[p^k], k\geq 1$. (The assumption that the group be linear is actually superfluous, but we will not need the result in this increased generality.)

Open subgroups as above yield an especially strong form of the Lie group-Lie algebra correspondence over the $p$-adics:

\begin{thm}[{$p$-adic Lie group-Lie algebra correspondence (see \cite[Theorem 3, p. 152,  and Corollary 1, p. 156]{serre-liebook})}]
    Two $p$-adic Lie groups have isomorphic Lie algebras if and only if they contain isomorphic open subgroups.

    Every finite-dimensional $p$-adic Lie algebra is realized by a $p$-adic Lie group.
\end{thm}
Note that in \cite{serre-liebook}, the language of analytic group chunks is used, but in the case of $p$-adics, every analytic group chunk contains an open subgroup (see \cite[Corollary 1, p. 117]{serre-liebook} and also \cite[Corollary, p. 141]{serre-liebook}), yielding the above simplified formulation of the correspondence.

A last fact we will need is the metric behavior of power maps $\mu_d: x\mapsto x^d, \, d\geq 1,$ near identity:

\begin{prop} Let $G$ be a linear $p$-adic group and $d\geq 1$. Then on a sufficiently small neighborhood $U$ of identity, the map $x\mapsto x^d$ is
\begin{enumerate}[(i)]
	\item a contraction if $p\mid d$, and
	\item an isometry if $p,d$ are coprime.
\end{enumerate}
\end{prop}
In fact, as the proof will show, it suffices to take $U$ to be the first congruence subgroup. The statement is also true more generally for all $p$-adic groups, but we will not need it.
\begin{proof} Write $G\subseteq \GL(N,\bbQ_p)$. We work with the induced metric from $M(N,\bbZ_p)$ given by the maximal distance between entries. Then an element with distance $p^{-k}$ to the identity is of the form $1 + p^k A$ for some $A\in M(N,\bbZ_p)$ that is not divisible by $p$. For (i), note that
	$$(1+p^k A)^p = 1 + p^{k+1} A + O(p^{2k}),$$
which is at distance $p^{-(k+1)}$ from identity. So the $p$th power map is a contraction.

On the other hand, for (ii), we have
	$$(1+p^k A)^d = 1 + p^k d A + O(p^{2k}),$$
which remains at distance $p^{-k}$ from the identity. So $d$th power preserves distance to the identity. \end{proof}}

\subsection{Notions of irreducibility}\label{subsec:irr}
In this subsection we prove Proposition \ref{prop:irred-vs-sirred} showing that irreducibility and strong irreducibility are equivalent.
Recall that for a subset of factors $J\subsetneq I$, we write $\pr_J:G\to G_J$ for the natural projection. Our goal will be to prove that projections of irreducible groups to proper subsets of factors do not merely fail to be discrete, but are in fact almost dense:

{
\begin{prop}\label{prop:dense projections}
 Let $\Gamma\subseteq G$ be discrete and irreducible. Then for any proper subset $J\subsetneq I$, the closure $\overline{\pr_J(\Gamma)}$contains $G_J^+$.
\end{prop}

We will often use the following observation without explicit comment:
\begin{rem}\label{rmk:+-vs-restrict} For a connected semisimple algebraic group $\bbH$ defined over a field $k$ of characteristic zero, passing from $\bbH(k)$ to $\bbH(k)^+$ commutes with restriction of scalars (see \cite[Section I.1.7]{margulis-book}), so that $(R_p G)^+ = G_p^+$. \end{rem}

\begin{proof}[Proof of Proposition \ref{prop:irred-vs-sirred} from Proposition \ref{prop:dense projections}] 
Recall that we defined $\Gamma$ to be \emph{strongly irreducible} if it has 
almost dense projections. Since $G_J^+\subseteq G_J$ has finite index by Theorem \ref{thm:G+}.(i), we see that Proposition \ref{prop:dense projections} shows that any irreducible group is strongly irreducible. Conversely, since for all $p$, $G_p^+$ is $\bbQ_p$-Zariski dense in $G_p$ by Theorem \ref{thm:G+}.(i), the projection of a strongly irreducible group to $G_p$ is $\bbQ_p$-Zariski dense (and indiscrete), and therefore any strongly irreducible group is, as the name suggests, irreducible. 
\end{proof}}

We start by proving Proposition \ref{prop:irred-vs-sirred} in the `pure' case where all factors of $J$ are analytic over the same field. We will actually prove the following stronger statement that will be used in the general case.

\begin{lemma}\label{lemma:irred-pure} Let $\Gamma\subseteq G$ be discrete and irreducible. Let $p$ be a finite prime or $\infty$ and let $H\subseteq G_p$ be a closed, indiscrete subgroup normalized by $\Gamma$. Then there is a subset $I_H\subseteq I_p$ such that $G_{I_H}^+$ is an open subgroup of $H$.\end{lemma}

\begin{proof} Since $H$ is a closed subgroup of the analytic group $G_p$, its Lie algebra $\frh$ is well-defined, and since $H$ is not discrete, $\frh$ is nontrivial. Since $\frh$ is $\Ad(\pr_{I_p}(\Gamma))$-invariant and $\pr_{I_p}(\Gamma)$ is $\bbQ_p$-Zariski dense in $G_p$, it follows that $\frh$ is an ideal in the Lie algebra of $G_p$. Hence there is a nonempty subset $I_H\subseteq I_p$ such that $\frh=\oplus_{i\in I_H} \frg_i$. We will show that $H$ contains $G_{I_H}^+$.

To start, note that $H\cap G_{I_H}$ is open and closed in $G_{I_H}$. Openness follows from the fact that $H \cap G_{I_H}$ has the same Lie algebra as
$G_{I_H}$. In the archimedean case, $H$ therefore contains the connected component of identity, which coincides with $G_{I_H}^+$. In the nonarchimedean case, it suffices to prove $H\cap G_i$ is noncompact for every $i\in I_H$, since then Theorem \ref{thm:tits-prasad} shows that $H\cap G_i$ contains $G_i^+$. To establish noncompactness, note that $H\cap G_i$ is normalized by $\pr_i(\Gamma)$, which is unbounded by Lemma \ref{rmk:irred->unbdd}).  { Now notice that any bounded open subgroup $O$ has bounded normalizer  by an application of the Bruhat-Tits fixed point theorem. Openness implies that the fixed set of $O$ in the Bruhat-Tits building is itself bounded and so the normalizer fixes its barycenter.   So  $H\cap G_i$ must be unbounded as well.} \end{proof}

\begin{rem} In the pure case, the the above lemma indeed shows irreducibility implies strong irreducibility: Letting $\Gamma$ be discrete and irreducible, $J\subsetneq I$, and taking $H:=\overline{\pr_J(\Gamma)}$, the above Lemma gives a subset $I_H\subseteq J$ such that $G_{I_H}^+\subseteq H$ is open. In fact we must have $I_H=J$: Since $H/G_{I_H}^+$ is finite, if there exists $j\in J\bs I_H$, then $\pr_j(H)$ and hence $\pr_j(\Gamma)$ would be discrete, but this contradicts irreducibility of $\Gamma$. This shows $G_J^+\subseteq \overline{\pr_J(\Gamma)}$ and hence the latter has finite index in $G_J$. \end{rem}

We will complete the proof of Proposition \ref{prop:irred-vs-sirred} in the general, possibly mixed, case.

\begin{proof} Write $J=\sqcup_{p\in P_J} J_p$ where $J_p$ consists of the $\bbQ_p$-analytic factors in $J$ and $P_J$ is a finite subset of primes and possibly $\infty$. Let $H$ denote the closure of $\pr_J(\Gamma)$. By passing to a finite index subgroup of $\Gamma$, we can assume that $H\subseteq G_J^+$ and we aim to show $H=G_J^+$. First, suppose that the connected component of identity $H^\circ$ of $H$ is nontrivial (so that necessarily $J_\infty\neq\varnothing$). We apply Lemma \ref{lemma:irred-pure} to the subgroup $H^\circ$ and obtain that there is a subset $J_{H,\infty}\subseteq J_\infty$ such that $G_{J_H,\infty}^+\subseteq H^\circ$ is open. Therefore we can project to the factors given by $J\bs J_{H,\infty}$ and assume $J_{H,\infty}=\varnothing$.

We will aim to show that after the above reduction, $P_J$ consists of finite primes and $H$ contains $G_{J_p}^+$ for every $p\in P_J$. Since $H$ is locally compact and totally disconnected, it contains a compact open subgroup $K$. Because $K$ is compact and totally disconnected, its image in $G_{J_\infty}$ is finite, so that by possibly shrinking $K$, we can assume $K$ projects trivially to the archimedean factors $G_{J_\infty}$ and therefore $H$ has discrete projection. However, this is impossible because the projection of $\Gamma\subseteq H$ to $G_{J_\infty}$ is indiscrete, so we must have $J_\infty=\varnothing$.

{
It remains to show that $H$ contains $G_{J_p}^+$ for every $p\in P_J$. For any integer $d$ define a map $\mu_d:x\mapsto x^d$ of $K$ to itself.  Let $K_p$ be the projection of $K$ to $G_{J_p}$ for $p\in P_J$. By possibly shrinking $K$ further, we can assume that the map $\mu_p$ is a contraction on $K_p$ for all $p\in P_J$. Then we claim that $K=\prod_{p\in P_J} K_p$. Indeed, let $k=(k_p)_p\in K\subseteq \prod_p K_p$. Fix $p$ and let $d$ be the product of the remaining primes. Since $\mu_d$ is an isometry on $K_p$ but a contraction on $K_\ell$ for all $\ell\neq p$, there is some sequence $n_m\to\infty$ such that $k_p^{d^{n_m}}=\mu_d^{n_m}(k_p)\to k_p$ and hence $k^{d^{n_m}}\to (k_p,e,\dots, e)\in K_p \times \prod_{\ell\neq p} K_\ell$. Since $K$ is closed, we see that $K_p\subseteq K$ for all $p\in P_J$, as desired.

Further, we claim that $K_p$ is nondiscrete for all $p\in P_J$. Indeed, if the image of $K$ in $G_{J_p}$ is discrete, then so is the image of $H$. This follows since $K$ contains the intersection of $\pi_{J_p}(\Gamma)$ with an open neighborhood of the identity in $G_{J_p}$. Since we know $\Gamma$ has nondiscrete image this is impossible. Since $K_p$ is nondiscrete and $K=\prod_p K_p$ is contained in $H$, it follows that for every $p$, the intersection $H\cap G_{J_p}$ is nondiscrete. We apply Lemma \ref{lemma:irred-pure} to the subgroup $H\cap G_{J_p}$ and obtain that for every $p\in P_J$, there is a subset $J_{H,p}\subseteq J_p$ such that $G_{J_H,p}^+\subseteq H$ is open.}

Set $J_H:=\sqcup_{p\in P_J} J_{H,p}$. To complete the proof, we show $J_H=J$: We have shown above that $G_{J_H}^+$ is an open subgroup of $H$, so the image of $H$ in the complementary factors $G_{J\bs J_H}$ is discrete. This contradicts irreducibility of $\Gamma$. \end{proof}

\subsection{A lemma of Venkataramana}
\label{subsec:venky}

{ 

In this subsection we prove Lemma \ref{lem-venky} below, which we have used in the proof of Proposition \ref{prop-irred2latt}.
 The argument below is similar to the proof of a lemma of Venkataramana's first stated as \cite[Lemma 2.3]{lubotzky-zimmer} and generalized in \cite[Lemma 4.1]{f-mj-vl}. The latter applies to subgroups $\Gamma$ of $S$-arithmetic lattices containing a $T$-arithmetic lattice for $T\subseteq S$. Here we need a slightly different version where $\Gamma$ is not given as a subgroup of a lattice, so we include it for completeness. Unlike the citations above, what we need here does not requires any approximation theorems.   We shall use the following notation that is based on the proof of Proposition~\ref{prop-irred2latt}:
 \begin{enumerate}
 \item $G$ is as in Proposition~\ref{prop-irred2latt}.
 \item $\Gamma\subseteq G$ is a discrete and irreducible subgroup.
 \item $G_{\text{na}}$ is a non-empty product of nonarchimedean factors of
 $G$.
 \item $G_{0}$ is the product of the other factors of $G$ and  assume that $G_{0}$ is non-trivial.
  \item $K_{\text{na}}\subseteq G_{\text{na}}^+$ is a maximal compact open subgroup.
     \item $\pr_{na}$ denotes the projection to $G_{\text{na}}$.
    \item $\Gamma_{\text{na}}\subseteq \Gamma$ denotes the subgroup that maps into $K_{\text{na}}$ under the projection $\pr_{na}$ to $G_{\text{na}}$.
    \item $\pr_{0}$ denotes the projection to $G_{0}$.
 \end{enumerate}

 \begin{lemma}\label{lem-venky}
 Let $G, \Gamma, G_{\text{na}}, G_0, K_{\text{na}}, \Gamma_{\text{na}}, \pr_{0}$ be as above. Assume that
 	\begin{itemize}
 		\item $\pr_{na}(\Gamma_{\text{na}})$ is dense in $K_{\text{na}}$.
 	\item $\pr_{0}(\Gamma_{\text{na}})$ is a lattice in $G_0$.
 	\end{itemize}
 	Then $\Gamma\subseteq G$ is a lattice.
 	\end{lemma}
 \begin{proof}
 By passing to a finite index subgroup, we can assume without loss of generality that $\Gamma\subseteq G^+$.
 Let $F\subseteq G_0^+$ be a fundamental domain for the lattice $\pr_0(\Gamma_{\text{na}})\subseteq G_0^+$. As in Venkataramana's lemma and \cite[Lemma 4.1]{f-mj-vl}, to show $\Gamma\subseteq G$ is a lattice, we will show that $F\times K_{\text{na}}$ is a fundamental domain for $\Gamma\subseteq G^+$.

 First we show the $\Gamma$-translates of $F\times K_{\text{na}}$ are disjoint: If $\gamma(F\times K_{\text{na}})\cap (F\times K_{\text{na}})\neq \varnothing$, then by projecting to the second factor and using that $K_{\text{na}}$ is a group, we see that $\gamma\in K_{\text{na}}$, and hence $\gamma\in \Gamma_{\text{na}}$. But then by projecting to the first factor, we see that $\gamma F\cap F\neq\varnothing$, and since $F$ is a fundamental domain for $\Gamma_{\text{na}}$, we conclude that $\gamma = e$.

 Next, we show that $\Gamma(F\times K_{\text{na}})=G$. Since $K_{\text{na}}$ is $\pr_{\text{na}}(\Gamma_{\text{na}})$-invariant and $F\subseteq G_0$ is a fundamental domain for $\pr_{0}(\Gamma_{\text{na}})$, it suffices to show that $\pr_{\text{na}} (\Gamma)K_{\text{na}} = G$. But this is immediate because $\pr_{\text{na}}(\Gamma)\subseteq G_{\text{na}}^+$ is dense and $K_{\text{na}}$ is open.
 	\end{proof}
}

\bibliographystyle{alpha}
\bibliography{ref}

\end{document}